\documentclass[11pt]{article}
\usepackage{amsmath, amsfonts, amssymb, amsthm,  graphicx, enumerate}
\usepackage[letterpaper, left=1truein, right=1truein, top = 1truein, bottom = 1truein]{geometry}
 
\usepackage{appendix}
\usepackage[numbers, square]{natbib}
\usepackage[
            CJKbookmarks=true,
            bookmarksnumbered=true,
            bookmarksopen=true,
            colorlinks=true,
            citecolor=red,
            linkcolor=blue,
            anchorcolor=red,
            urlcolor=blue
            ]{hyperref}
 \usepackage[usenames]{color}

\newcommand{\black}{\color{black}}
\newcommand{\rateconst}{36} 

\newcommand{\Ex}{\mathsf{E}}
\renewcommand{\Pr}{\mathsf{P}}

\newcommand{\cov}{\mathsf{Cov}}
\newcommand{\tr}{\mathsf{Tr}}

\newcommand{\wh}{\widehat}

\newcommand{\Asel}{\wh{A}}
\newcommand{\Ssel}{\S_{\Asel}}

\newcommand{\BB}{\mathbf{B}}

\newcommand{\CC}{\mathbf{C}}

\newcommand{\Hyper}{{\rm Hypergeometric}}

\renewcommand{\u}{\mathbf{u}}
\renewcommand{\v}{\mathbf{v}}
\newcommand{\w}{\mathbf{w}}

\newcommand{\A}{\mathbf{A}}
\newcommand{\B}{\mathbf{B}}
\newcommand{\C}{\mathbf{C}}
\newcommand{\E}{\mathbf{E}}
\newcommand{\J}{\mathbf{J}}

\newcommand{\M}{\mathbf{M}}
\newcommand{\W}{\mathbf{W}}
\newcommand{\X}{\mathbf{X}}
\newcommand{\Z}{\mathbf{Z}}

\newcommand{\Y}{\mathbf{Y}}
\newcommand{\x}{\mathbf{x}}

\newcommand{\LL}{\boldsymbol{\Lambda}}

\newcommand{\tE}{\widetilde{\E}}
\newcommand{\hE}{\widehat{\E}}
\newcommand{\bfSigma}{\boldsymbol{\Sigma}}

\newcommand{\hSS}{\widehat{\bfSigma}}
\newcommand{\OO}{\boldsymbol{\Omega}}

\newcommand{\hOO}{\widehat{\OO}}

\newcommand{\U}{\mathbf{U}}
\newcommand{\D}{\mathbf{D}}
\newcommand{\V}{\mathbf{V}}

\renewcommand{\S}{\mathbf{S}}

\newcommand{\I}{\mathbf{I}}

\newcommand{\col}[2]{{#1}_{* {#2}}}
\newcommand{\row}[2]{{#1}_{{#2} *}}

\newcommand{\ind}[1]{{\mathbf{1}_{\left\{{#1}\right\}}}}
\newcommand{\supp}{\mathrm{supp}}
\renewcommand{\sp}{\mathrm{span}}

\newcommand{\hf}{\frac{1}{2}}

\newcommand{\RR}{\mathbb{R}}

\theoremstyle{plain}
\newtheorem{theorem}{Theorem}

\newtheorem{lemma}{Lemma}

\theoremstyle{definition}

\newtheorem{remark}{Remark}

\newcommand{\beq}{\begin{equation}}
\newcommand{\eeq}{\end{equation}}
\newcommand{\beas}{\begin{eqnarray*}}
\newcommand{\eeas}{\end{eqnarray*}}
\newcommand{\bea}{\begin{eqnarray}}
\newcommand{\eea}{\end{eqnarray}}
\newcommand{\bei}{\begin{itemize}}
\newcommand{\eei}{\end{itemize}}
\newcommand{\ben}{\begin{enumerate}}
\newcommand{\een}{\end{enumerate}}
 
\newcommand{\hV}{\wh\V}
\newcommand{\hv}{\wh\v}

\usepackage{prettyref,soul,xspace}

\newcommand{\eg}{e.g.\xspace}
\newcommand{\ie}{i.e.\xspace}

\newcommand{\ones}{\mathbf 1}
\newcommand{\reals}{{\mathbb{R}}}
\newcommand{\integers}{{\mathbb{Z}}}
\newcommand{\naturals}{{\mathbb{N}}}

\newcommand{\eexp}{{\rm e}}

\newcommand{\diff}{{\rm d}}

\newcommand{\rank}{\mathop{\sf rank}}
\newcommand{\Tr}{\mathop{\sf Tr}}
\newcommand{\diag}{\mathop{\text{diag}}}

\newcommand{\Expect}{\mathsf{E}}
\newcommand{\expect}[1]{\mathsf{E} #1 }

\newcommand{\Prob}{\mathsf{P}}
\newcommand{\prob}[1]{{ \mathsf{P}\left\{ #1 \right\} }}

\newcommand{\eqdistr}{{\stackrel{\rm (d)}{=}}}

\newcommand{\argmax}{\mathop{\rm argmax}}

\newcommand{\Th}{{^{\rm th}}}

\newtheorem{prop}{Proposition}

\newrefformat{eq}{(\ref{#1})}
\newrefformat{chap}{Chapter~\ref{#1}}
\newrefformat{sec}{Section~\ref{#1}}
\newrefformat{algo}{Algorithm~\ref{#1}}
\newrefformat{fig}{Fig.~\ref{#1}}
\newrefformat{tab}{Table~\ref{#1}}
\newrefformat{rmk}{Remark~\ref{#1}}
\newrefformat{clm}{Claim~\ref{#1}}
\newrefformat{def}{Definition~\ref{#1}}
\newrefformat{cor}{Corollary~\ref{#1}}
\newrefformat{lmm}{Lemma~\ref{#1}}
\newrefformat{lemma}{Lemma~\ref{#1}}
\newrefformat{prop}{Proposition~\ref{#1}}
\newrefformat{app}{Appendix~\ref{#1}}
\newrefformat{ex}{Example~\ref{#1}}
\newrefformat{exer}{Exercise~\ref{#1}}
\newrefformat{soln}{Solution~\ref{#1}}

\newcommand{\lunder}[1]{{\underset{\raise0.3em\hbox{$\smash{\scriptscriptstyle-}$}}{#1}}}

\newcommand{\norm}[1]{\|{#1} \|}

\newcommand{\fnorm}[1]{\|#1\|_{\rm F}}
\newcommand{\opnorm}[1]{\| #1 \|}
\newcommand{\Opnorm}[1]{\left\| #1 \right\|}

\newcommand{\iprod}[2]{\left \langle #1, #2 \right\rangle}
\newcommand{\Iprod}[2]{\langle #1, #2 \rangle}

\newcommand{\indc}[1]{{\mathbf{1}_{\left\{{#1}\right\}}}}

\def\innergetnumber#1[#2]#3{#2}
\def\getnumber{\expandafter\innergetnumber\jobname}

\newcommand{\bszero}{{\boldsymbol{0}}}

\newcommand{\bbB}{{\mathbb{B}}}

\newcommand{\bbS}{{\mathbb{S}}}

\newcommand{\calE}{{\mathcal{E}}}

\newcommand{\bfv}{{\mathbf{v}}}

\newcommand{\bfA}{{\mathbf{A}}}

\newcommand{\bfI}{{\mathbf{I}}}

\newcommand{\bfV}{{\mathbf{V}}}
\newcommand{\bfW}{{\mathbf{W}}}
\newcommand{\bfX}{{\mathbf{X}}}
\newcommand{\bfY}{{\mathbf{Y}}}
\newcommand{\bfZ}{{\mathbf{Z}}}

\newcommand{\tu}{{\tilde{u}}}

\newcommand{\tI}{{\tilde{I}}}

\newcommand{\comp}[1]{{#1^{\rm c}}}

\newcommand{\ntok}[2]{{#1,\ldots,#2}}

\newcommand{\pth}[1]{\left( #1 \right)}

\newcommand{\sth}[1]{\left\{ #1 \right\}}

\newcommand{\fracd}[2]{\frac{\diff #1}{\diff #2}}

\renewcommand{\tu}{{\tilde{\u}}}
\newcommand{\tw}{{\tilde{\w}}}

\title{Optimal Estimation and Rank Detection for Sparse Spiked Covariance Matrices}
\author{
Tony Cai\footnote{Tony Cai is with Department of Statistics, The Wharton School, University of Pennsylvania, Philadelphia, PA 19104.
The research of Tony Cai was supported in part by NSF FRG Grant DMS-0854973, NSF Grant DMS-1208982, 
and NIH Grant R01 CA 127334-05.},
~~ Zongming Ma\footnote{Zongming Ma is with Department of Statistics, The Wharton School, University of Pennsylvania, Philadelphia, PA 19104. The research of Zongming Ma is supported in part by the Dean's Research Fund of the Wharton School.}, ~ and  ~ Yihong Wu\footnote{Yihong Wu is with Department of Electrical and Computer Engineering, University of Illinois at Urbana-Champaign, Urbana, IL 61801. The research of  Yihong Wu was supported in part by NSF FRG Grant DMS-0854973.}\\
\\
}

\date{}

\begin{document}
	\maketitle

\begin{abstract}
This paper considers sparse spiked covariance matrix models in the high-dimensional setting and studies the minimax estimation of the covariance matrix and the principal subspace as well as the minimax rank detection. 
The optimal rate of convergence for estimating the spiked covariance matrix under the spectral norm is established, which requires significantly different techniques from those for estimating other structured covariance matrices such as bandable or sparse covariance matrices. 
We also establish the minimax rate under the spectral norm for estimating the principal subspace, the primary object of interest in principal component analysis. In addition,  the optimal rate for the rank detection boundary is obtained. 
This result also resolves a gap between the upper and lower bounds in a recent paper by \citet{BR12} where the special case of rank one is considered.  

\end{abstract} 	
 


\newpage

\section{Introduction}

Covariance matrix plays a fundamental role in multivariate analysis. Many  methodologies, including discriminant analysis,  principal component analysis and clustering analysis, rely critically on the knowledge of the covariance structure. Driven by a wide range of contemporary applications in many fields including genomics, signal processing, and financial econometrics,  estimation of covariance matrices in the high-dimensional setting is of particular interest. 

There have been significant recent advances  on the estimation of a large covariance matrix and its inverse, the precision matrix. 
A variety of regularization methods, including banding, tapering, thresholding and penalization, have been introduced for estimating several classes of covariance and precision matrices with different structures. See, for example,  
\cite{BJ08b, BJ08a,
CL11, CZZ10,
CZ12, FHT08, Karoui08, LamFan09, 
RWRY11, Yuan10},
among many others.
 

\subsection{Sparse spiked covariance matrix model}
In the present paper, we consider spiked covariance matrix models in the high-dimensional setting, which arise naturally from \emph{factor models} with homoscedastic noise. 
To be concrete, suppose that we observe an $n\times p$ data matrix $\X$  with the rows $\row{\X}{1}$, ..., $\row{\X}{n}$ i.i.d.~following a multivariate normal distribution with mean $\bszero$ and covariance matrix $\bfSigma$, denoted by by $N(\bszero,\bfSigma)$,
where the covariance matrix $\bfSigma$ is given by 
\begin{equation}
\bfSigma = \cov(\row{\X}{i})= \V \LL \V' + \sigma^2 \I_p,
	\label{eq:spike-model}
\end{equation}
where $\LL = {\rm diag}(\lambda_1,\dots, \lambda_r)$ with $\lambda_1\geq \cdots\geq \lambda_r >0$, and $\V = [\bfv_1,\dots,\bfv_r]$ is $p\times r$ with orthonormal columns. The $r$ largest eigenvalues of $\bfSigma$ are $\lambda_i+\sigma^2$, $i=1, ..., r$, and the rest are all equal to $\sigma^2$. The $r$ leading eigenvectors of $\bfSigma$ are given by the column vectors of $\V$. Since the spectrum of $\bfSigma$ has $r$ spikes, \eqref{eq:spike-model} is termed by \cite{Johnstone01} as the spiked covariance matrix model.
This covariance structure and its variations have been widely used in signal processing, chemometrics, econometrics, population genetics, and many other fields.
See, for instance, \cite{Fan08, Kritchman08, onat09, patterson}.
In the high-dimensional setting, various aspects of this model have been studied by several recent papers, including but not limited to \cite{BR12, Birnbaum12, CMW12, JohnstoneLu09, Jung09, Ma11, Onatski12, Paul07}.
For simplicity, we assume $\sigma$ is known. 
Since  $\sigma$  can always be factored out by scaling $\X$, without loss of generality, we assume $\sigma = 1$. Data-based estimation of $\sigma$ will be discussed in \prettyref{sec:discussion}.
 
The primary focus of this paper is on the setting where $\V$ and $\bfSigma$ are sparse, and our goal is threefold. 
First, we consider the minimax estimation of the spiked covariance matrix $\bfSigma$ under the spectral norm. The method as well as the optimal rates of convergence in this problem are considerably different from those for estimating other recently studied structured covariance matrices, such as bandable and sparse covariance matrices. 
Second, we are interested in rank detection. The rank $r$ plays an important role in principal component analysis (PCA) and is also of significant interest in signal processing and other applications. 
Last but not least, we consider optimal estimation of the principal subspace $\sp(\V)$ under the spectral norm, which is the main object of interest in PCA. Each of these three problems is important in its own right. 

We now explain the sparsity model of $\V$ and $\bfSigma$. 
The difficulty of estimation and rank detection depends on the joint sparsity of the columns of $\V$.
Let $\row{\V}{j}$ denote the $j\Th$ row of $\V$. 
The row support of $\bfV$ is defined by
\begin{equation}
\label{eq:supp}
\supp(\V) = \{j\in [p]: \row{\V}{j} \neq \bszero\},
\end{equation}
{\black whose cardinality is denoted by $|\supp(\V)|$.}
Let the collection of $p \times r$ matrices with orthonormal columns be
$O(p,r) = \sth{\V \in \reals^{p \times r}: \V'\V=\I_r}$.
Define the following parameter spaces for $\bfSigma$,
\begin{equation}
\label{eq:para.space}
\begin{aligned}
\Theta_0(k,p,r,\lambda,\tau) = 
\{\bfSigma = \V \LL \V' + \I_p: & \lambda/\tau \leq \lambda_r\leq\cdots\leq\lambda_1 \leq \lambda,\\
&  \V \in O(p,r),\, |\supp(\V)| \leq k\},	
\end{aligned}
\end{equation}
where $\tau\geq 1$ is a constant {\black and $r\leq k\leq p$ is assumed throughout the paper.} Note that {\black the condition number of $\LL$ is at most ${\tau}$}. 
Moreover, for each covariance matrix in $\Theta_0(k,p,r,\lambda,\tau)$, the leading $r$ singular vectors (columns of $\V$) are \emph{jointly $k$-sparse} in the sense that 
the row support size of $\V$ is upper bounded by $k$. 
The structure of group sparsity has proved useful for high-dimensional regression; See, for example, \cite{Lounici11}. 
In addition to \prettyref{eq:para.space}, we also define the following parameter spaces by dropping the dependence on $\tau$ and $r$, respectively: 
\begin{equation}
\begin{aligned}
\Theta_1(k,p,r,\lambda) 
= & ~ \mathbf{cl} \bigcup_{\tau \geq 1} \Theta_0(k,p,r,\lambda,\tau) \\
= & ~ \{\bfSigma = \V \LL \V' + \I_p: 0 \leq \lambda_r\leq\cdots\leq \lambda_1 \leq \lambda, \\
& \hskip 9em \V \in O(p,r),\, |\supp(\V)| \leq k \}	
\end{aligned}
\label{para.space1}
\end{equation}
and
\begin{equation}
\label{para.space2}
\Theta_2(k,p, \lambda,\tau)  = \bigcup_{r = 0}^k \Theta_0(k,p,r,\lambda,\tau).
\end{equation}
As a consequence of the group sparsity in $\V$, a covariance matrix $\bfSigma$ in any of the above parameter spaces
has at most $k$ rows and $k$ columns containing nonzero off-diagonal entries.
We note that the matrix is more structured than the so-called ``$k$-sparse'' matrices considered in \cite{BJ08b, CL11,  CZ12}, where each row (or column) has at most $k$ nonzero off-diagonals. 

\subsection{Main contributions}
In statistical decision theory, the minimax rate quantifies the difficulty of an inference
problem and is frequently used as a benchmark for the performance of inference
procedures.
The main contributions of this paper include the \emph{sharp non-asymptotic} minimax rates for estimating the covariance matrix $\bfSigma $ and the principal subspace $\sp(\V)$ under the squared spectral norm loss, as well as for detecting the rank $r$ of the principal subspace. In addition, we also establish the minimax rates for estimating the precision matrix $\OO = \bfSigma^{-1}$ as well as the eigenvalues of $\bfSigma$ under the spiked covariance matrix model \eqref{eq:spike-model}. 

We establish the minimax rate for estimating the spiked covariance matrix $\bfSigma$ in \eqref{eq:spike-model} under the spectral norm
\begin{equation}
\label{eq:sigma-loss}
L(\bfSigma,\wh\bfSigma) = \norm{\bfSigma -  \wh\bfSigma}^2,
\end{equation}
where for a matrix $\A$ its spectral norm is defined as $\|\A\|=\sup_{\|\x\|_2=1}\| \A\x\|_2$ {\black with $\|\cdot\|_2$ the vector $\ell_2$ norm}. 
The minimax upper and lower bounds developed in Sections \ref{sec:upper} and \ref{sec:lower} yield the following optimal rate for estimating sparse spiked covariance matrices under the spectral norm
\begin{equation}
\inf_{\hSS}\sup_{\bfSigma \in\Theta_1(k,p,r,\lambda) } \Ex \|\hSS -\bfSigma \|^2 
\asymp
{\black
\bigg [
\frac{(\lambda+1) k}{n} \log \frac{\eexp p}{k} + \frac{\lambda^2 r}{n}
\bigg ] 
}  \wedge \lambda^2,
\label{eq:minimax-rate-Sigma}
\end{equation}
subject to certain mild regularity conditions.\footnote{Here and after, {\black $a\wedge b \triangleq \min(a,b)$} and $a_n \asymp b_n$ means that $\frac{a_n}{b_n}$ is bounded from both below and above by constants independent of $n$ and all model parameters.}
{\black The two terms in the squared bracket are contributed by the estimation error of the eigenvectors $\V$ and the eigenvalues, respectively. Note that the second term can be dominant if $\lambda$ is large.}

An important quantity of {\black the spiked model} is the rank $r$ of the principal subspace $\sp(\V)$, or equivalently, the number of spikes in the spectrum of $\bfSigma$, which is of significant interest in chemometrics \cite{Kritchman08}, signal array processing \cite{BN09}, and other applications. 
Our second goal is 
the minimax estimation of the rank $r$ under \emph{the zero--one loss}, or equivalently, the minimax detection of the rank $r$. 
It is intuitively clear that the difficulty in estimating the rank $r$ depends crucially on the magnitude of the minimum spike $\lambda_r$. 
Results in Sections \ref{sec:upper} and \ref{sec:lower} show that the optimal rank detection boundary over the parameter space $\Theta_2(k,p, \lambda,\tau)$ is of order $\sqrt{\frac{k}{n} \log \frac{\eexp p}{k}}$. 
Equivalently, the rank $r$ can be 
{\black exactly recovered with high probability}
if $\lambda_r \geq \beta \sqrt{\frac{k}{n} \log \frac{e p}{k}}$ for a sufficiently large constant $\beta$; On the other hand, reliable detection becomes impossible by any method if $\lambda_r \leq \beta_0 \sqrt{\frac{k}{n} \log \frac{e p}{k}}$ for some positive constant $\beta_0$.
{\black Lying in the heart of the arguments is a careful analysis of the moment generating function of a squared symmetric random walk stopped at a hypergeometrically distributed time, which is summarized in \prettyref{lmm:rm}.
It is worth noting that the optimal rate for rank detection obtained in the current paper resolves a gap left open in a recent paper by Berthet and Rigollet \cite{BR12}, where the authors obtained the optimal detection rate $\sqrt{\frac{k}{n} \log \frac{e p}{k}}$ for the rank-one case in the regime of $k \ll \sqrt{p}$, but the lower bound deteriorates to $\sqrt{\frac{p}{kn}}$ when $k \gg \sqrt{p}$ which is strictly suboptimal.
}

{\black In many statistical applications, instead of the covariance matrix itself, the object of direct interest is often a lower dimensional functional of the covariance matrix, \eg, the principal subspace $\sp(\V)$.
} This problem is known in the literature as sparse PCA \cite{Birnbaum12, CMW12, JohnstoneLu09,Ma11}.
The third goal of the paper is the minimax estimation of the principal subspace $\sp(\V)$.
To this end, we note that the principal subspace can be uniquely identified with the associated projection matrix $\V\V'$. Moreover, any estimator can be identified with a projection matrix $\wh\V\wh\V'$, where the columns of $\wh\V$ constitute an orthonormal basis for the subspace estimator.
Thus, estimating $\sp(\V)$ is equivalent to estimating $\V\V'$. We aim to optimally estimate $\sp(\V)$ under the loss \cite[Section II.4]{stsu90}
\begin{equation}
\label{eq:loss} 
L(\V,\wh\V) = \norm{\V\V' - \wh\V\wh\V'}^2,
\end{equation}
which equals the squared sine of the largest canonical angle between the respective linear spans. In the sparse PCA literature, the loss \prettyref{eq:loss} was first used in \cite{Ma11} {\black for multi-dimensional subspaces}.
For this problem, we shall show that,
under certain regularity conditions, the minimax rate of convergence is
\begin{equation}
\inf_{\hV} \sup_{\bfSigma\in\Theta_0(k,p, r, \lambda,\tau)} \Ex \|\hV\hV' -\V\V' \|^2 
\asymp
 \frac{ (\lambda+1)k}{n \lambda^2} \log \frac{\eexp p}{k} \wedge 1.
	\label{eq:minimax-rate-V}
\end{equation}

{\black 
In the present paper we considered estimation of the principal subspace $\sp(\V)$ under the spectral norm loss \eqref{eq:loss}. It is interesting to compare the results with those for optimal estimation under the Frobenius norm loss \cite{CMW12,Vu12space} $\fnorm{\V\V' - \wh\V\wh\V'}^2$, whose ratio to \prettyref{eq:loss} is between $1$ and $2r$. The optimal rate under the spectral norm loss given in \eqref{eq:minimax-rate-V} does not depend on the rank $r$, whereas the optimal rate under the Frobenius norm loss has an extra term $\frac{\lambda+1}{n\lambda^2}r(k-r)$, which depends on the rank $r $ quadratically through $r(k - r)$ \cite{CMW12}.
Therefore the rate under the Frobenius norm far exceeds \prettyref{eq:minimax-rate-V} when $r\gg \log\frac{\eexp p}{k}$.
When $r = 1$, both norms lead to the same rate and the result in \eqref{eq:minimax-rate-V} recovers earlier results on estimating the leading eigenvector obtained in \cite{Birnbaum12, Vu12,Lounici12}.
}

In addition to the optimal rates for estimating the  covariance matrix $\bfSigma$, the rank $r$ and the principal subspace $\sp(\V)$,  the minimax rates for estimating the precision matrix $\OO = \bfSigma^{-1}$ as well as the eigenvalues of $\bfSigma$ are also established. 

\subsection{Other related work}
\label{sec:other}
Apart from the spiked covariance matrix model studied in this paper, other covariance matrix models have been considered in the literature. The most commonly imposed structural assumptions include ``Toeplitz", where  each descending diagonal from left to right is constant,  ``bandable", where the entries of the covariance matrix decay as they move away from the diagonal, and ``sparse", where only a small number of entries in each row/column are nonzero. 
The optimal rates of convergence were established in \cite{CaiRenZhou12}, \cite{CZZ10} and \cite{CZ12} for estimating Toeplitz, bandable, and sparse covariance matrices, respectively.
Estimation of sparse precision matrices has also been actively studied due to its close connection to Gaussian graphical models \cite{CLZ12,RWRY11,Yuan10}. 
{\black In addition, our work is also connected to the estimation of effective low-rank covariance matrices. See, for example, \cite{Lounici12low, Bunea12} and the reference therein.}

\subsection{Organization}
The rest of the paper is organized as the following. \prettyref{sec:procedure} introduces basic notation and then gives a detailed description of the procedures for estimating the spiked covariance matrix $\bfSigma$, the rank $r$ and the principal subspace $\sp(\V)$. The rates of convergence of these estimators are given in Section \ref{sec:upper}. 
Section \ref{sec:lower} presents the minimax lower bounds that match the upper bounds in Section \ref{sec:upper} in terms of the convergence rates, thereby establishing the minimax rates of convergence and rate-optimality of the estimators constructed in Section \ref{sec:procedure}. The minimax rates for estimating the eigenvalues and the precision matrix are given in Section \ref{sec:precision-eigenvalue}. Section \ref{sec:discussion} discusses computational and other related issues. The proofs are given in Section \ref{sec:proof}.


\section{Estimation Procedure}
\label{sec:procedure}

We give a detailed description of the estimation procedure in this section and study its properties in Section \ref{sec:upper}.  
Throughout, we shall focus on minimax estimation and assume the sparsity $k$ is
known, while the rank $r$ will be selected based on data. 
Adaptation to $k$ will be discussed in Section \ref{sec:discussion}.

\paragraph{Notation}
We first introduce some notation. For any matrix $\X = (x_{ij})$ and any vector $\u$, denote by $\|\X\|$ the spectral norm, $\fnorm{\X}$ the Frobenius norm, and $\|\u\|$ the vector $\ell_2$ norm. 
Moreover,  the $i\Th$ row of $\X$ is denoted by $\row{\X}{i}$ and the $j\Th$ column by $\col{\X}{j}$.
Let $\supp(\X) = \{i: \row{\X}{i} \neq 0\}$ denote the row support of $\X$. For a positive integer $p$, $[p]$ denotes the index set $\{1, 2, ..., p\}$. 
For any set $A$, $|A|$ denotes its cardinality, and $\comp{A}$ its complement.
For two subsets $I$ and $J$ of indices, denote by $\X_{IJ}$ the $|I|\times |J|$ submatrices formed by $x_{ij}$ with $(i,j) \in I \times J$. Let $\row{\X}{I} = \X_{I [n]}$ and $\col{\X}{J} = \X_{[p] J}$.
For any square matrix $\A = (a_{ij})$, we let $\Tr(\A) = \sum_{i}a_{ii}$ be its trace. Define the inner product of any two matrices $\BB$ and $\CC$ of the same size by  $\Iprod{\BB}{\CC} = \Tr(\BB'\CC)$.
For any matrix $\A$, we use $\sigma_i(\A)$ to denote its $i\Th$ largest singular value. When $\A$ is positive semi-definite, $\sigma_i(\A)$ is also the $i\Th$ largest eigenvalue of $\A$.
For any real number $a$ and $b$, set $a\vee b = \max\{a,b\}$ and $a\wedge b = \min\{a,b\}$. 
Let $\bbS^{p-1}$ denote the unit sphere in $\reals^p$. 
For any event $E$, we write 
$\indc{E}$
as its indicator function.

For any set $B\subset [p]$, let $\comp{B}$ be its complement.
For any symmetric matrix $\bfA\in \RR^{p\times p}$, we use $\bfA_B$ to denote the $p\times p$ matrix whose $B\times B$ block is $\bfA_{BB}$, the remaining diagonal elements are all ones and the remaining off-diagonal elements are all zeros, \ie, 
\begin{equation}
(\bfA_B)_{ij} = a_{ij} \indc{i \in B,j\in B} + \indc{i=j \in \comp{B}}.	
	\label{eq:AB}
\end{equation}
In other word, after {\black proper} reordering of rows and columns, we have
\begin{equation*}
\bfA_B = \begin{bmatrix}
\bfA_{BB} & \bszero \\ \bszero & \bfI_{\comp{B}\comp{B}}
\end{bmatrix}.
\end{equation*}


Let $P \otimes Q$ denote the product measure of $P$ and $Q$ and $P^{\otimes n}$ the $n$-fold product of $P$.
For random variables $X$ and $Y$, we write $X\eqdistr Y$ if they follow the same distribution, and $X \stackrel{\rm{s.t.}}{\leq} Y$ if $\Pr(X>t) \leq \Pr(Y>t)$ for all $t\in \reals$.
Throughout the paper, we use $C$ to denote a generic positive absolute constant, whose actual value may depend on the context. 
For any two sequences $\{a_n\}$ and $\{b_n\}$ of positive numbers, we write $a_n\lesssim b_n$ when $a_n\leq C b_n$ for some numeric constant $C$, and $a_n\gtrsim b_n$ when $b_n \lesssim a_n$, and $a_n \asymp b_n$ when both $a_n\gtrsim b_n$ and $a_n\lesssim b_n$ hold.


\paragraph{Estimators}


We are now ready to present the procedure for estimating the spiked covariance matrix $\bfSigma$, the rank $r$, and the principal subspace $\sp(\V)$.

Let
\begin{equation}
A = \supp(\V)
\end{equation}
be the row support of $\V$. 
For any $m\in [p]$, let
\begin{equation}
\label{eq:diag-dev}
\eta(m,n,p,\gamma_1) = 2\left(\sqrt{\frac{m}{n}} + \sqrt{ \frac{\gamma_1}{n} \,m\log\frac{\eexp p}{m}} \right) + \left(\sqrt{\frac{m}{n}} + \sqrt{ \frac{\gamma_1}{n} \,m\log\frac{\eexp p}{m}} \right)^2.
\end{equation}

Recall that the observed matrix $\bfX$ has i.i.d.~rows $\row{\bfX}{i}\sim N(\bszero, \bfSigma)$. We define $\S = \frac{1}{n}\bfX'\bfX$ as the sample covariance matrix.
Also recall that we assume knowledge of the sparsity level $k$ which is an upper bound for the support size $|A|$.
The first step in the estimation scheme is to select a subset $\wh{A}$ of $k$ features based on the data.
To this end, let
\begin{equation}
\label{eq:supp-set}
\begin{aligned}
\bbB_k = \{
B\subset [p]:~ & |B| = k,~~\mbox{and for all $D\subset \comp{B}$ with $|D|\leq k$},\\
& \hskip -4em \norm{\S_{D}-\bfI}\leq \eta(|D|,n,p,\gamma_1),~
\norm{\S_{DB}}\leq 2\sqrt{\norm{\S_{B}}}\, \eta(k,n,p,\gamma_1)
\}.	
\end{aligned}
\end{equation}
The appropriate value of $\gamma_1$ will be specified later in the statement of the theorems.
Intuitively speaking, the requirements in \eqref{eq:supp-set} aim to ensure that 
for any $B\in \bbB_k$, there is no evidence in data suggesting that $\comp{B}$ overlaps with the row support $A$ of $\V$.
If $\bbB_k\neq \emptyset$, denote by $\Asel$ an arbitrary element of $\bbB_k$ (or we can let $\Asel = \argmax_{B\in \bbB_k} \tr(\S_{BB})$ for concreteness).
As we will show later, $\bbB_k$ is non-empty with high probability; See \prettyref{lem:inclusion} in \prettyref{sec:pf-lmm}.
The set $\Asel$ represents the collection of selected features, which turns out to be instrumental in constructing optimal estimators for the three objects we are interested in: 
the covariance matrix, 
the rank of the spiked model, and
the principle subspace.
The estimator $\Asel$ of the support set $A$ is obtained through searching over all subsets of size $k$. Such a global search, though computationally expensive, appears to be necessary in order for our procedure to optimally estimate $\bfSigma$ and $\V$ under the spectral norm. For example, estimating row-by-row is not guaranteed to always yield optimal results.  Whether there exist computationally efficient procedures attaining the optimal rate is currently unknown. See Section \ref{sec:discussion} for more discussions.

Given $\Asel$, the estimators for the above three objects are defined as follows. 
Recalling the notation in \prettyref{eq:AB}, we define the covariance matrix estimator as
\begin{equation}
	\wh\bfSigma = \S_{\Asel} \indc{\bbB_k \neq \emptyset} + \I_p \indc{\bbB_k = \emptyset},
	\label{eq:hSS}
\end{equation}
The estimator for the rank is
\begin{equation}
\label{eq:r-hat}
\wh{r} = \max\Big\{l: \sigma_l(\S_{\Asel \Asel}) \geq 1 + 
\gamma_2\,\sqrt{\opnorm{\Ssel}}\, 
\eta(k,n,p,\gamma_1) 
\Big\}.
\end{equation}
The appropriate value of $\gamma_2$ will be specified later in the statement of the theorems.
Last but not least,
the estimator for the principle subspace is $\sp(\hV)$, where
\begin{equation}
\label{eq:V-hat}
\hV = [\wh\v_1,\dots,\wh\v_{\hat{r}}]\in O(p,\wh{r})
\end{equation}
with $\wh\v_l$ the $l\Th$ eigenvector of $\wh\bfSigma$.
When $\bbB_k = \emptyset$, we set $\wh{r} = 0$ and $\wh\V = 0$ since $\wh\bfSigma = \I_p$.
Note that the estimator $\wh\V$ is based on the estimated rank $\wh{r}$.
Whenever $\wh{r}\neq r$, the value of the loss function \eqref{eq:loss} equals $1$.

For a brief discussion on the comparison between the foregoing estimation procedure and that in \cite{Vu12space}, see \prettyref{sec:discussion}.

\section{Minimax Upper Bounds}
\label{sec:upper}

We now investigate the properties of the estimation procedure introduced in Section \ref{sec:procedure}. Rates of convergence for estimating the whole covariance matrix and its principal subspace under the spectral norm as well as for rank detection are established. The minimax lower bounds given in Section \ref{sec:lower} will show that these  rates are optimal and together they yield the minimax rates of convergence.

%

We begin with the estimation error of the covariance matrix estimator \eqref{eq:hSS}.
Note that here we do not require the ratio $\lambda_1/\lambda_r$ to be bounded.

\begin{theorem}
\label{thm:rate-Sigma}
Let $\wh\bfSigma$ be defined in \eqref{eq:hSS} with $\bbB_k$ given by \eqref{eq:supp-set} for some $\gamma_1 \geq 10$.
If $\frac{k}{n}\log\frac{\eexp p}{k}\leq c_0$ for a sufficiently small constant $c_0 > 0$, 
then 
\begin{equation}
\sup_{\bfSigma \in\Theta_1(k, p, r, \lambda)} \Ex \|\hSS -\bfSigma \|^2 
\lesssim
\frac{(\lambda+1) k}{n} \log \frac{\eexp p}{k} + \frac{\lambda^2 r}{n},
\label{eq:rate-Sigma}
\end{equation}
where the parameter space $\Theta_1(k,p,r,\lambda)$ is defined in \eqref{para.space1}.
\end{theorem}

As we shall show later in \prettyref{thm:lb-Sigma-V}, the rates in \eqref{eq:rate-Sigma} are optimal with respect to all the model parameters, namely $k$, $p$, $r$ and $\lambda$.

\begin{remark}
Since the parameter space $\Theta_1$ in \eqref{para.space1} is contained in the set of $k$-sparse covariance matrices, it is of interest to compare the minimax rates for covariance matrix estimation in these two nested parameter spaces. 
For simplicity, only consider the case where the spectral norms of covariance matrices are uniformly bounded by a constant.
Cai and Zhou \cite{CZ12} showed that, under certain regularity conditions, the minimax rate of convergence for estimating $k$-sparse matrices is 	
$k^2 \frac{\log p}{n}$,
while the rate over $\Theta_1$ in \eqref{eq:minimax-rate-Sigma} reduces to $\frac{k}{n}\log\frac{\eexp p}{k}$ when the spectral norm of the matrix and hence $\lambda$ is bounded. 
Thus, ignoring the logarithmic terms, the rate over the smaller parameter space $\Theta_1$ is faster by a factor of $k$. 
This faster rate can be achieved because the group $k$-sparsity considered in our parameter space imposes much more structure than the row-wise $k$-sparsity does for the general $k$-sparse matrices. 
\end{remark}


The next result concerns with the detection rate of the rank estimator \eqref{eq:r-hat} under the extra assumption that the ratio of the largest spike to the smallest, i.e., $\lambda_1/\lambda_r$, is bounded.

\begin{theorem}
	\label{thm:rate-rank}
Let  $\hat{r}=\hat{r}(\gamma_1,\gamma_2)$ be defined in \prettyref{eq:r-hat} for some constants $\gamma_1 \geq 10$ and $\gamma_2 \geq 8\sqrt{\gamma_1}+34$. 
Assume that
\begin{equation}
	\label{eq:thm-rank-assumpt}
\frac{\tau k}{n} \log \frac{\eexp p}{k} \leq c_0 
\end{equation}
for a sufficiently small constant $c_0\in (0,1)$ which depends on $\gamma_1$.
If $\lambda \geq \beta \sqrt{\frac{k}{n} \log \frac{\eexp p}{k}}$ for some sufficiently large $\beta$ depending only on $\gamma_1,\gamma_2$ and $\tau$, then
\begin{equation}
\sup_{\bfSigma \in\Theta_2(k,p, \lambda,\tau)} \prob{\wh{r} \neq \rank(\LL)} \leq 
C  p^{1-\gamma_1/2}
\label{eq:rank-uppbd}
\end{equation}
where the parameter space $\Theta_2(k,p,\lambda,\tau)$ is defined in \eqref{para.space2}.
\end{theorem}

By \prettyref{thm:lb-rank} to be introduced later, the detection rate of $\sqrt{\frac{k}{n}\log\frac{\eexp p}{k}}$ is optimal.
For more details, see the discussion in \prettyref{sec:lower}.

Finally, we turn to the risk of the principal subspace estimator. 
As in \prettyref{thm:rate-rank}, we require $\lambda_1/\lambda_r$ to be bounded.

\begin{theorem}
		\label{thm:rate-V}
Suppose 
\begin{align}
\label{eq:n-p-lambda-assumpt}
M_1\log{p} \geq \log{n} \geq M_0\log\lambda
\end{align}
holds for some absolute constants $M_0, M_1 > 0$.
Let
$\wh\V$ be defined in \eqref{eq:V-hat} with constants $\gamma_1 \geq 10\vee (1+\frac{2}{M_0})M_1$ and $\gamma_2\geq 8\sqrt{\gamma_1}+34$ in \eqref{eq:supp-set} and \eqref{eq:r-hat}.
If \eqref{eq:thm-rank-assumpt} holds for a sufficiently small constant $c_0$ depending on $\gamma_1$, then
\begin{equation}
\sup_{\bfSigma\in\Theta_0(k,p,r,\lambda,\tau)} \Ex \|\hV\hV' -\V\V' \|^2 
\lesssim
 \frac{k (\lambda+1)}{n \lambda^2} \log \frac{\eexp p}{k} \wedge 1,
	\label{eq:rate-V}
\end{equation}
where the parameter space $\Theta_0(k,p,r,\lambda,\tau)$ is defined in \eqref{eq:para.space}.
\end{theorem}
\begin{remark}
To ensure that the choice of $\gamma_1$ for achieving \eqref{eq:rate-V} is data-driven, we only need an under-estimate for $M_0 = \log{n}/\log{\lambda}$, or equivalently an over-estimate for $\lambda$.
(Note that $M_1 = \log{n}/\log{p}$ can be obtained directly given the data matrix.)
To this end, we first estimate $\bfSigma$ by $\hSS$ in \eqref{eq:hSS} with an initial $\gamma_1 = 10$ in \eqref{eq:supp-set}.
Then we control $\lambda$ by $2\opnorm{\hSS}-1$. By the proof of \prettyref{thm:rate-rank}, and in particular \eqref{eq:Ssel-norm}, this is an over-estimate of $\lambda$ with high probability. 
The upper bound in \eqref{eq:rate-V} remains valid if we compute $\wh{\V}$ with a (possibly) new $\gamma_1 = 10\vee (1+2/\wh{M_0})M_1$ in \eqref{eq:supp-set}, where $\wh{M_0} = \log{n}/ \log(2\opnorm{\hSS}-1)$.
\end{remark}

It is worth noting that the rate in \eqref{eq:rate-V} does not depend on $r$, and is optimal, by the lower bound given in \prettyref{thm:lb-Sigma-V} later.

The problems of estimating $\bfSigma$ and $\V$ are clearly related, but they are also distinct from each other. To discuss their relationship, we first note the following result (proved in \prettyref{app:SS-V}) which is a variation of the well-known sin-theta theorem \cite{Davis70}:
\begin{prop}
\label{prop:SS-V}
Let $\bfSigma$ and $\hSS$ be $p \times p$ symmetric matrices. Let $r \in [p-1]$ be arbitrary and let $\V,\hV \in O(p,r)$ be formed by the $r$ leading singular vectors of $\bfSigma$ and $\hSS$, respectively. Then
\begin{equation}
	\norm{\hSS-\bfSigma} \geq \frac{1}{2} (\sigma_r(\bfSigma)- \sigma_{r+1}(\bfSigma)) \norm{\hV\hV'-\V\V'}.
	\label{eq:SS-V}
\end{equation}
\end{prop}

In view of \prettyref{prop:SS-V}, the minimax risks for estimating the spiked covariance matrix $\bfSigma$ and the principle subspace $\V$ under the spectral norm can be tied as follows:
\begin{equation}
{\black \inf_{\hSS} \sup_{\bfSigma \in\Theta} \Ex \|\hSS -\bfSigma \|^2  \gtrsim
\lambda_r^2 \inf_{\hV} \sup_{\bfSigma \in\Theta} \Ex \|\hV\hV' -\V\V' \|^2 } ,
\label{eq:SS-V-1}	
\end{equation}
where $\Theta=\Theta_0(k,p,r,\lambda,\tau)$.

The results of Theorems \ref{thm:rate-Sigma} and \ref{thm:rate-V} suggest, however, that the above inequality is not tight when $\lambda$ is large.
The optimal rate for estimating $\V$ is not equivalent to that for estimating $\bfSigma$ divided by $\lambda^2$ when $\frac{\lambda+1}{\lambda^2} \log \frac{\eexp p}{k} \ll \frac{r}{k}$. 
Consequently, \prettyref{thm:rate-V} cannot be directly deduced from \prettyref{thm:rate-Sigma} but requires a different analysis by introducing an intermediate matrix $\S_0$ defined later in \eqref{eq:S0}. 
{\black This is because the estimation of $\bfSigma$ needs to take into account the extra error in estimating the eigenvalues in addition to those in estimating $\V$.}
On the other hand, in proving \prettyref{thm:rate-Sigma} we need to contend with the difficulty that the loss function is unbounded.


\section{Minimax Lower Bounds and Optimal Rates of Convergence}
\label{sec:lower} 

In this section we derive minimax lower bounds for estimating the spiked covariance matrix $\bfSigma$ and the principal subspace $\sp(\V)$  as well as for the rank detection. These lower bounds hold for all parameters and are non-asymptotic.
The lower bounds together with the upper bounds given in Section \ref{sec:upper} establish the optimal rates of convergence for the three problems.

The technical analysis heavily relies on a careful study of a  rank-one testing problem  and analyzing the moment generating function of a squared symmetric random walk stopped at a hypergeometrically distributed time. This lower bound technique is of independent interest and can be useful for other related matrix estimation and testing problems.
 

\subsection{Lower bounds and minimax rates for matrix and subspace estimation}

We first consider the lower bounds  for estimating the spiked covariance matrix $\bfSigma$ and the principal subspace $\V$ under the spectral norm.

\begin{theorem}
For any $1\leq r\leq k\leq p$ and $n \in \naturals$,
\begin{equation}
\label{eq:lb-Sigma}
\inf_{\widetilde{\bfSigma}} \sup_{\bfSigma \in \Theta_1(k,p,r,\lambda)} \Ex \|\widetilde{\bfSigma}-\bfSigma \|^2 
\gtrsim
 \pth{\frac{(\lambda+1) k}{n} \log \frac{\eexp p}{k} + \frac{\lambda^2 r}{n}}  \wedge \lambda^2,
\end{equation}
and
\begin{equation}
\label{eq:lb-V}
\inf_{\widetilde{\V}} \sup_{\bfSigma\in \Theta_0(k,p,r,\lambda,\tau)} \Ex \|\widetilde{\V}\widetilde{\V}' -\V\V' \|^2 
\gtrsim
 \frac{(\lambda+1)(k-r)}{n \lambda^2} \log \frac{\eexp (p-r)}{k-r} \wedge 1,
\end{equation}
where the parameter spaces $\Theta_1(k,p,r,\lambda)$ and $\Theta_0(k,p,r,\lambda,\tau)$ are defined in \eqref{para.space1} and \eqref{eq:para.space}, respectively.
\label{thm:lb-Sigma-V}
\end{theorem}

To better understand the lower bound \eqref{eq:lb-Sigma}, it is helpful to write it equivalently  as 
\[
\pth{\frac{(\lambda+1) k}{n} \log \frac{\eexp p}{k}}  \wedge \lambda^2 + \frac{\lambda^2 r}{n},  
\]
which can be proved by showing that the minimax risk is lower bounded by each of these two terms. The first term does not depend on $r$ and is the same as the lower bound in the rank-one case. The second term is the oracle risk when the true support of $\V$ is known. The key to the proof is the analysis of the rank-one case which will be discussed in more detail in Section \ref{sec:rank-one}. The proof of \eqref{eq:lb-V} is relatively straightforward by using known results on rank-one estimation.

In view of the upper bounds given in Theorems \ref{thm:rate-Sigma} and \ref{thm:rate-V} and the lower bounds given in Theorem \ref{thm:lb-Sigma-V}, we establish the following minimax rates of convergence for estimating the spiked covariance matrix $\bfSigma$ and the principal subspace $\V$, subject to the constraints on the parameters given in Theorems \ref{thm:rate-Sigma} and \ref{thm:rate-V}:
\begin{eqnarray}
\inf_{\widetilde{\bfSigma}} \sup_{\bfSigma \in \Theta_1(k,p, r, \lambda,\tau)} \Ex \|\widetilde{\bfSigma}-\bfSigma \|^2 
&\asymp&
 \pth{\frac{(\lambda+1) k}{n} \log \frac{\eexp p}{k} + \frac{\lambda^2 r}{n}}  \wedge \lambda^2,
 \label{eq:minimax-Sigma}\\
\inf_{\widetilde{\V}} \sup_{\bfSigma\in\Theta_0(k,p, r, \lambda,\tau)} \Ex \|\widetilde{\V}\widetilde{\V}' -\V\V' \|^2
&\asymp&
 \frac{(\lambda+1) k}{n \lambda^2} \log \frac{\eexp p}{k} \wedge 1,
  \label{eq:minimax-V}
\end{eqnarray}
where \prettyref{eq:minimax-V} holds under the addition condition that $k -r \gtrsim k$. 
Therefore, the estimators of $\bfSigma$ and $\V$ given in Section \ref{sec:procedure} are rate optimal. 
In \eqref{eq:minimax-Sigma}, the trivial upper bound of $\lambda^2$ can always be achieved by using the identity matrix as the estimator.

\subsection{Lower bound and minimax rate  for rank detection}

We now turn to the lower bound and minimax rate  for the rank detection problem.

\begin{theorem}
\label{thm:lb-rank}
Let $\beta_0 < \frac{1}{36}$ be a constant. 
For any $1\leq r\leq k\leq p$ and $n \in \naturals$,
 if $\lambda \leq  1 \wedge \sqrt{\frac{\beta_0 k}{n} \log \frac{e p}{k}}$, then
	\begin{equation}
	\inf_{\tilde{r}} \sup_{\bfSigma \in \Theta_2(k,p,\lambda,\tau)} \prob{\tilde{r} \neq \rank(\bfSigma)} \geq w(\beta_0),
	\label{eq:rank-error}
\end{equation}
where the function $w: (0,\frac{1}{36}) \to (0,1)$ satisfies $w(0+)=1$. 
\end{theorem}
The upper and lower bounds given in Theorems  \ref{thm:rate-rank} and \ref{thm:lb-rank} show that the optimal detection boundary for the rank $r$ is $\sqrt{\frac{k}{n} \log \frac{e p}{k}}$. That is,  the rank $r$ can be estimated with an arbitrarily small error probability when $\lambda_r \geq \beta \sqrt{\frac{k}{n} \log \frac{e p}{k}}$ for a sufficiently large constant $\beta$, whereas this is impossible to achieve by any method if $\lambda_r \leq  \sqrt{\frac{\beta_0 k}{n} \log \frac{e p}{k}}$ for some small positive constant $\beta_0$. 
Note that \prettyref{thm:lb-rank} applies to the full range of sparsity including the non-sparse case $k=p$, which requires $\lambda_r \geq \sqrt{\frac{p}{n}}$. This observation turns out to be useful in proving the ``parametric term'' in the minimax lower bound for estimating $\bfSigma$ in \prettyref{thm:lb-Sigma-V}.


The rank detection lower bound in \prettyref{thm:lb-rank} is in fact a direct consequence of the next proposition concerning testing the identity covariance matrix against rank-one alternatives, 
	\begin{equation}
H_{0}:  \bfSigma = \I_p, \quad \mbox{versus}\quad
H_{1}:  \bfSigma = \I_p + \lambda \v \v', \quad \v \in \bbS^{p-1} \cap B_0(k),
	\label{eq:HT.rank01}
\end{equation}
where $B_0(k) \triangleq \{\x \in \reals^p: |\supp(\x)| \leq k\}$. Note that $\bfSigma$ is in the parameter space $\Theta_2$ under both the null and the alternative hypotheses. The rank-one testing problem \eqref{eq:HT.rank01} has been studied in \cite{BR12}, where there is a gap between the lower and upper bounds when $k\gtrsim \sqrt{p}$. The following result show that their lower bound is in fact sub-optimal in this case. We shall give below a dimension-free lower bound for the optimal probability of error and determine the optimal rate of separation. The proof is deferred to \prettyref{sec:pf-lower}.
\begin{prop}
\label{prop:rank01}
Let $\beta_0 < \frac{1}{36}$ be a constant. 
Let $\X$ be an $n \times p$ random matrix whose rows are independently drawn from $N(0,\bfSigma)$. 
For any $k \in [p]$ and $n \in \naturals$, if $\lambda \leq 1 \wedge \sqrt{\frac{\beta_0 k}{n} \log \frac{\eexp p}{k}}$, the minimax sum of Type-I and Type-II error probabilities for the testing problem \eqref{eq:HT.rank01} satisfies
\[
\calE_n(k,p,\lambda) \triangleq
\inf_{\phi: \reals^{n \times p} \to \{0,1\}} \pth{ \Prob_{\I_p}\{\phi(\X)=1\} + \sup_{\bfSigma \in H_1} \Prob_{\bfSigma}\{\phi(\X)=0\} }  \geq w(\beta_0)
\]
where the function $w: (0,\frac{1}{36}) \to (0,1)$ satisfies $w(0+)=1$. 
\end{prop}

\prettyref{prop:rank01} shows that testing independence in the rank-one spiked covariance model can be achieved reliably only if the \emph{effective signal-to-noise-ratio} 
\[
\beta_0 = \frac{\lambda^2 }{\frac{k}{n} \log \frac{\eexp p}{k}}\to \infty.
\]
Furthermore, the lower bound in \prettyref{prop:rank01} also captures the following phenomenon:
if $\beta_0$ vanishes, then the optimal probability of error converges to one.
In fact, the lower bound in \prettyref{prop:rank01} is optimal in the sense that the following test succeeds with vanishing probability of error if $\beta_0 \to \infty$:
\[
\text{reject $H_0$ if and only if } \max_{\substack{J \subset [p]\\ |J|=k}} \|\S_{JJ} - \I_{JJ}\| \geq c \sqrt{\frac{k}{n} \log \frac{\eexp p}{k}}
\]
for some $c$ only depends on $\beta_0$. See, \eg, \cite[Section 4]{BR12}. However, the above test has high computational complexity since one needs to enumerate all $k\times k$ submatrices of $\S$. It remains an open problem to construct tests that are both computationally feasible and  minimax rate-optimal.

 

\subsection{Testing rank-one spiked model}
\label{sec:rank-one}

As mentioned earlier, a careful study of the rank-one testing problem \prettyref{eq:HT.rank01} provides a major tool for the lower bound arguments.
A key step in this study is the analysis of the moment generating function of a squared symmetric random walk stopped at a hypergeometrically distributed time. We present the main ideas in this section as the techniques can also be useful for other related matrix estimation and testing problems.

It is well-known that the minimax risk is given by the least-favorable Bayesian risk under mild regularity conditions on the model \cite{LeCam86}. 
For the composite testing problem \prettyref{eq:HT.rank01}, it turns out that the rate-optimal least-favorable prior for $\v$ is given by the distribution of the following random vector:
\begin{equation}
	\u = \frac{1}{\sqrt{k}} \J_I \w, 
	\label{eq:prior-wu}
\end{equation}
where 
$\w = (\ntok{w_1}{w_p})$ consists of iid Rademacher entries,
and $\J_I$ is a diagonal matrix given by $(\J_I)_{ii}=\indc{i \in I}$ with $I$ uniformly chosen from all subsets of $[p]$ of size $k$. In other words, $\u$ is uniformly distributed on the collection of $k$-sparse vectors of unit length with equal-magnitude non-zeros.
Hence $\u \in \bbS^{p-1} \cap B_0(k)$. We set
\[
\lambda^2 = \frac{\beta_0 k}{n} \log \frac{\eexp p}{k},
\]
where $\beta_0>0$ is a sufficiently small absolute constant.\footnote{Here $\beta_0$ can be chosen to be any constant smaller than $\frac{1}{36}$. See \prettyref{prop:rank01}. The number $\frac{1}{36}$ is certainly not optimized.}
The desired lower bound then follows if we establish that the following (Bayesian) hypotheses 
\begin{equation}
H_0: \bfSigma = \I_p \quad \text{v.s.} \quad H_1: \bfSigma = \I_p + \lambda \u\u'	
	\label{eq:rank01-u}
\end{equation}
cannot be separated with vanishing probability of error. 

\begin{remark}
	The composite testing problem \prettyref{eq:rank01-u} has also been considered in \cite{BR12}. In particular, the following suboptimal lower bound is given in \cite[Theorem 5.1]{BR12}: If
	\begin{equation}
	\lambda \leq 1 \wedge \sqrt{\frac{v k}{n}  \log\pth{1+ \frac{p}{k^2}}}
	\label{eq:RL}
\end{equation}
then the optimal error probability satisfies $\calE_n(k,p,\lambda) \geq C(v)$, where $C(v) \to 1$ as $v \to 0$. This result is established based on the following prior: 
\begin{equation}
\u = \frac{1}{\sqrt{k}} \J_I \ones,	
	\label{eq:prior-BR}
\end{equation}
which is a binary sparse vector with uniformly chosen support.

Compared to the result in \prettyref{prop:rank01}, \prettyref{eq:RL} is rate-optimal  
in the \emph{very sparse} regime of $k = o(\sqrt{p})$. 
However, since $\log(1+x) \asymp x$ when $x \lesssim 1$,  in the \emph{moderately sparse} regime of $k \gtrsim \sqrt{p}$, 
$\log(1+\frac{p}{k^2}) \asymp \frac{p}{k^2}$ and so the lower bound in \eqref{eq:RL} is substantially smaller than the optimal rate in \prettyref{prop:rank01} by a factor of $\sqrt{\frac{k^2}{p}\log\frac{\eexp p}{k}}$, which is a \emph{polynomial} factor in $k$ when $k \gtrsim p^\alpha$ for any $\alpha > 1/2$.
In fact, by strengthening the proof in \cite{BR12}, one can show that the optimal separation for discriminating \prettyref{eq:rank01-u} using the binary prior \prettyref{eq:prior-BR} is $1 \wedge \sqrt{\frac{k}{n}\log \frac{\eexp p}{k}} \wedge \frac{p}{k \sqrt{n}}$.
 Therefore the prior \prettyref{eq:prior-BR} is rate-optimal \emph{only} in the regime of $k = o(p^{\frac{2}{3}})$, while \prettyref{eq:prior-wu} is rate-optimal for all $k$. Examining the role of the prior \prettyref{eq:prior-wu} in the proof of \prettyref{thm:lb-rank}, we see that it is necessary to randomize the signs of the singular vector in order to take advantage of the central limit theorem and Gaussian approximation. When $k\gtrsim p^{\frac{2}{3}}$, the fact that the singular vector $\u$ is positive componentwise reduces the difficulty of the testing problem.

\end{remark}

The main technical tool for establishing the rank-detection lower bound in \prettyref{prop:rank01} is the following lemma, which can be of independent interest. It deals with the behavior of a symmetric random walk stopped after a hypergeometrically distributed number of steps. 
Moreover, note that \prettyref{lmm:rm} also incorporates the non-sparse case ($k=p$ and $H=k$), which proves to be useful in establishing the minimax lower bound for estimating $\bfSigma$ in \prettyref{thm:lb-Sigma-V}.  The proof of
\prettyref{lmm:rm} is deferred to \prettyref{sec:pf-lmmrm}.
\begin{lemma}
	Let $p \in \naturals$ and $k \in [p]$. Let $\ntok{B_1}{B_k}$ be independently Rademacher distributed. Denote the symmetric random walk on $\integers$ stopped at the $m\Th$ step by
	\begin{equation}
	G_m \triangleq \sum_{i=1}^m B_i.
	\label{eq:Gm}
\end{equation}
Let $H \sim \Hyper(p,k,k)$ with $\prob{H=i} = \frac{\binom{k}{i} \binom{p-k}{k-i}}{\binom{p}{k}}, i = \ntok{0}{k}$. 
Then there exists a function $g: (0,\frac{1}{36}) \to (1,\infty)$ with $g(0+)=1$,
 such that for any $a < \frac{1}{36}$,
\begin{equation}
\expect{\exp\pth{t G_H^2}} \leq g(a)\, ,
	\label{eq:rm}
\end{equation}
where $t = \frac{a}{k} \log \frac{\eexp p}{k}$.
	\label{lmm:rm}
\end{lemma}	

\begin{remark}[Tightness of \prettyref{lmm:rm}]
	The purpose of \prettyref{lmm:rm} is to seek the largest $t$, as a function of $(p,k)$, such that $\expect{\exp\pth{t G_H^2}}$ is upper bounded by a constant non-asymptotically. The condition that $t \asymp \frac{1}{k} \log \frac{ep}{k}$ is in fact both necessary and sufficient. To see the necessity, note that $\prob{G_H = H | H=i} = 2^{-i}$. Therefore 
	\[
	\expect{\exp\pth{t G_H^2}} \geq \expect{\exp(t H^2) 2^{-H}} \geq \exp(t k^2) 2^{-k} \, \prob{H=k} \geq  \exp\pth{tk^2 - k\log \frac{2p}{k}},
	\]
which cannot be upper bounded by an absolutely constant unless $t \lesssim \frac{1}{k} \log \frac{\eexp p}{k}$.

	\label{rmk:tightlmm}
\end{remark}


\section{Estimation of Precision Matrix and Eigenvalues}
\label{sec:precision-eigenvalue}

We have so far focused on the optimal rates for estimating the spiked covariance matrix $\bfSigma$, the rank $r$ and the principal subspace $\sp(\V)$ . 
The technical results and tools developed in the earlier sections turn out to be readily useful for establishing the optimal rates of convergence for estimating the precision matrix $\OO = \bfSigma^{-1}$ as well as the eigenvalues of $\bfSigma$ under the spiked covariance matrix model \eqref{eq:spike-model}. 

Besides the covariance matrix $\bfSigma$, it is often of significant interest to estimate the precision matrix $\OO$, which is closely connected to the Gaussian graphical model as the support of the precision matrix $\OO$ coincides with the edges in the corresponding Gaussian graph. 
Let $\hSS$ be defined in \eqref{eq:hSS} and let $\sigma_i(\hSS)$ denote its $i\Th$ largest eigenvalue value for all $i\in [p]$. 
Define the precision matrix estimator as
\begin{equation}
\wh\OO = \begin{cases}
(\hSS)^{-1}, & \sigma_p(\hSS) \geq \frac{1}{2}, \\
\I_p, & \sigma_p(\hSS) < \frac{1}{2}.
\end{cases}
\label{eq:hOO}
\end{equation}
The following result gives the optimal rates for estimating $\OO$ under the spectral norm.

\begin{prop}[Precision matrix estimation]
\label{prop:rate-Omega}
Assume that $\lambda \asymp 1$. If $\frac{k}{n}\log\frac{\eexp p}{k}\leq c_0$ for a sufficiently small constant $c_0 > 0$, then
\begin{equation}
\label{eq:rate-Omega}
\inf_{\hOO} \sup_{\bfSigma \in \Theta_1(k,p,r,\lambda)} \Ex \opnorm{\hOO - \OO}^2 
\asymp  \frac{k}{n}\log\frac{\eexp p}{k},
\end{equation}
where the upper bound is attained by the estimator \prettyref{eq:hOO} with $\gamma_1\geq 10$ in obtaining $\hSS$.
\end{prop}

The upper bound follows from the lines in the proof of \prettyref{thm:rate-Sigma} after we control the smallest eigenvalue of $\hSS$ as in \eqref{eq:hOO}.
\prettyref{prop:rank01} can be readily applied to yield the desired lower bound.

Note that the optimal rate in \eqref{eq:rate-Omega} is quite different from the minimax rate of convergence 
$M^2k^2\frac{\log p}{n}$
for estimating $k$-sparse precision matrices where each row/column has at most $k$ nonzero entries. Here $M$ is the $\ell_1$ norm bound for the precision matrices. See \cite{CLZ12}. So the sparsity in the principal eigenvectors and the sparsity in the precision matrix itself have significantly different implications for estimation of $\OO$ under the spectral norm.

We now turn to estimation of the eigenvalues. Since $\sigma$ is assumed to be equal to one, it suffices to estimate the eigenvalue matrix
 $\E = \diag(\lambda_i)$ where $\lambda_i \triangleq 0$ for $i > r$. 
For any estimator $\tE = \diag(\tilde{\lambda}_i)$, we quantify the estimator error by the loss function $\|\tE -\E \|^2 = \max_{i\in[p]} |\tilde{\lambda}_i - \lambda_i|^2$.
The following result gives the optimal rate of convergence for this estimation problem.
\begin{prop}[Uniform estimation of spectra]
	\label{prop:rate-Lambda}
	Under the same conditions of \prettyref{prop:rate-Omega}, 
		\begin{equation}
\inf_{\hE} \sup_{\bfSigma\in\Theta_1(k,p,r,\lambda)} \Ex \|\hE -\E \|^2
\asymp 
\frac{k }{n} \log \frac{\eexp p}{k},
	\label{eq:rate-LL}
\end{equation}
where the upper bound is attained by the estimator $\hE = \diag(\sigma_i(\hSS)) - \I_p$ with $\hSS$ defined in \prettyref{eq:hSS} with $\gamma_1 \geq 10$.
\end{prop}
Hence, the spikes and the eigenvalues can be estimated uniformly at the rate of $\frac{k}{n} \log \frac{\eexp p}{k}$ when $k\ll n/\log p$. \prettyref{prop:rate-Lambda} is a direct consequence of \prettyref{thm:rate-Sigma} and \prettyref{prop:rank01}. The proofs of Propositions \ref{prop:rate-Omega} and \ref{prop:rate-Lambda} are deferred to \prettyref{sec:pf-cor}.

\section{Discussions}
\label{sec:discussion}


We have assumed the knowledge of the noise level $\sigma = 1$ and the support size $k$. 
For a given value of $\sigma^2$, one can always rescale the data and reduce to the case of unit variance. As a consequence of rescaling, the results in this paper remain valid for a general $\sigma^2$ by replacing each $\lambda$ with $\lambda/\sigma^2$ in both the expressions of the rates and the definitions of the parameter spaces. 
When $\sigma^2$ is unknown, it can be easily estimated based on the data. 
Under the sparsity models \eqref{eq:para.space}--\eqref{para.space2}, when $k < p/2$,  $\sigma^2$ can be well estimated by $\wh{\sigma}^2 = \mathrm{median}(s_{jj})$ as suggested in \cite{JohnstoneLu09}, where $s_{jj}$ is $j\Th$ diagonal element of the sample covariance matrix. 

The knowledge of the support size $k$ is much more crucial for our procedure.
An interesting topic for future research is the construction of adaptive estimators which could achieve the minimax rates in Theorems \ref{thm:rate-Sigma}--\ref{thm:rate-V} without knowing $k$.
One possibility is to define $\bbB_k$ in \eqref{eq:supp-set} for all $k\in [p]$, find the smallest $k$ such that $\bbB_k$ is non-empty, and then define estimators for that particular $k$ similar to those in \prettyref{sec:procedure} with necessary adjustments accounting for the extra multiplicity in the support selection procedure.

{\black
For ease of exposition, we have assumed that the data are normally distributed and so various non-asymptotic tail bounds used in the proof follow.
Since these bounds typically only rely on sub-Gaussianity assumptions on the data, we expect the results in the present paper readily extendable to data generated from distributions with appropriate sub-Gaussian tail conditions.
}
The estimation procedure in \prettyref{sec:procedure} is different from that in \cite{Vu12space}. Although both are based on enumerating all possible support sets of size $k$, Vu and Lei \cite{Vu12space} proposed to pick the support set which maximizes a quadratic form, while ours is based on picking the set satisfying certain deviation bounds.

A more important issue is the computational complexity required to obtain the minimax rate optimal estimators. The procedure described in \prettyref{sec:procedure} entails a global search for the support set $A$, which can be computationally intensive. In many cases, {\black this seems unavoidable since} the spectral norm is not separable in terms of the entries/rows/columns. However, in some other cases, there are estimators that are computationally more efficient and can attain the same rates of convergence. 
For instance, in the low rank cases where $r \lesssim \log\frac{\eexp p}{k}$, the minimax rates for estimating $\sp(\V)$ under the spectral norm and under the Frobenius norm coincide with each other. 
See the discussion following \prettyref{thm:rate-V} in \prettyref{sec:upper}.
Therefore the procedure introduced in \cite[Section 3]{CMW12} attains the optimal rates under both norms simultaneously. 
As shown in \cite{CMW12}, this procedure is not only computationally efficient, but also adaptive to the sparsity $k$.
Finding the general {\black complexity-theoretic limits}
for attaining the minimax rates under the spectral norm is an interesting and challenging topic for future research.
{\black Following the initial post of the current paper, Berthet and Rigollet \cite{Berthet13} showed in a closely related sparse principal component detection problem that the minimax detection rate cannot be achieved by any computationally efficient algorithm} {\black in the highly sparse regime.}

\section{Proofs}
\label{sec:proof}

We first collect a few useful technical lemmas in Section \ref{sec:pf-lmm} before proving the main theorems in Section \ref{sec:pf-main} in the order of Theorems \ref{thm:rate-Sigma} - \ref{thm:lb-Sigma-V}. We then give the proofs of the propositions in the order of Propositions \ref{prop:rank01}, \ref{prop:rate-Omega}, \ref{prop:rate-Lambda}, and \ref{prop:SS-V}. 
As mentioned in Section \ref{sec:lower}, \prettyref{thm:lb-rank} 
on the rank detection lower bound  
is a direct consequence of  Proposition \ref{prop:rank01}.  We complete this section with the proof of \prettyref{lmm:rm}.

Recall that the row vector of $\X$ are i.i.d.~samples from the $N(\bszero, \bfSigma)$ distribution with $\bfSigma$ specified by \eqref{eq:spike-model}.
Equivalently, one can think of $\X$ as an $n\times p$ data matrix generated as
\begin{equation}
\label{eq:model}
	\X = \U \D \V' + \Z,
\end{equation}
where $\U$ is an $n\times r$ random effects matrix with iid $N(0,1)$ entries, $\D = {\rm diag}(\lambda_1^{1/2},\dots, \lambda_r^{1/2})$,
$\V$ is $p\times r$ orthonormal, and $\Z$ has iid $N(0,1)$ entries which are independent of $\U$.

\subsection{Technical Lemmas}
	\label{sec:pf-lmm}

\begin{lemma}
\label{lem:submat-opnorm-bd}
Let $\S = \frac{1}{n}\X'\X$ be the sample covariance matrix, then
\[
\Pr
\bigg( \max_{|B|=k} \norm{\S_{BB}} \leq (\lambda_1+1)
\Big(1 +\sqrt{\frac{k}{n}}+\frac{t}{\sqrt{n}}\Big)^2 \bigg) 
\leq \binom{p}{k} \eexp^{-t^2/2}.
\]
\end{lemma}
\begin{proof}
Note that
\begin{align*}
\max_{|B|=k}\norm{\S_{BB}} 
& \leq \max_{|B|=k} \norm{\bfSigma_{BB}}\cdot \norm{\bfSigma_{BB}^{-1/2}\S_{BB}\bfSigma_{BB}^{-1/2}}\\
& \leq (\lambda_1+1)\max_{|B|=k}
\norm{\bfSigma_{BB}^{-1/2}\S_{BB}\bfSigma_{BB}^{-1/2}}
\eqdistr (\lambda_1+1)
\max_{|B|=k} \norm{\bfZ_{*B}}^2.
\end{align*}
The result follows from the Davidson--Szarek bound \cite[Theorem II.7]{Davidson01} and the union bound.
 \end{proof}

\begin{lemma}
\label{lem:inclusion}
Let $\bbB_k$ be defined as in \prettyref{eq:supp-set} with $\gamma_1 \geq 3$. Then $\Pr(A\notin \bbB_k) \leq 5(\eexp p)^{1-\gamma_1/2}$.
\end{lemma}
\begin{proof}
Note that by union bound
\begin{align*}
\Pr(A\notin \bbB_k) & \leq
\Pr(\exists D\subset \comp{A}, |D|\leq k, \norm{\S_D-\bfI} > \eta(|D|,n,p,\gamma_1))\\
& ~~~ + \Pr(\exists D\subset \comp{A}, |D|\leq k, \norm{\S_{DA}} > 2\sqrt{\norm{\S_A}}\,\eta(k,n,p,\gamma_1)).
\end{align*}
We now bound the two terms on the right-hand side separately.

For the first term, note that $A=\supp(\V)$. Then for any $D \subset \comp{A}$, $\S_D = \frac{1}{n} \Z_{*D}' \Z_{*D}$. Hence
\begin{align*}
& \hskip -2em \Pr(\exists D\subset \comp{A}, |D|\leq k, \norm{\S_D-\bfI} > \eta(|D|,n,p,\gamma_1)) \\
& \leq \sum_{D\subset \comp{A}, |D|\leq k} \Pr(\norm{\S_D-\bfI} > \eta(|D|,n,p,\gamma_1)) \\
& \leq \sum_{l=1}^{k} \binom{p-k}{l} 2 \exp\left(-\frac{\gamma_1}{2}l \log\frac{\eexp p}{l}\right)\\
& \leq 2 \sum_{l=1}^k \left(\frac{\eexp p}{l}\right)^{l(1-\gamma_1/2)}
\leq 4(\eexp p)^{1-\gamma_1/2},
\end{align*}
where the second inequality is \cite[Proposition 4]{CMW12}, and the last inequality holds for all $\gamma_1\geq 3$ and $p\geq 2$.

For the second term, note that for any fixed $D\subset \comp{A}$,
$\S_{DA} = \frac{1}{n}\Z_{*D}'\X_{*A}$, where $\Z_{*D}$ and $\X_{*A}$ are independent. 
Thus, let $\bfW$ be the left singular vector matrix of $\X_{*A}$, we obtain that\footnote{If $\rank(\X_{*A}) < k$, then $\eqdistr$ is changed to $\stackrel{\rm{s.t.}}{\leq}$, and the subsequent arguments continue to hold verbatim.}
\[
\norm{\S_{DA}} \leq \frac{1}{n} \norm{\Z_{*D}'\bfW} \norm{\X_{*A}} 
\eqdistr \frac{1}{\sqrt{n}}\norm{\bfY} \sqrt{\norm{\S_A}},
\]
where $\bfY$ is a $|D|\times k$ matrix with i.i.d.~$N(0,1)$ entries.
Thus, we have
\begin{align*}
& \hskip -2em \Pr(\exists D\subset \comp{A}, |D|\leq k, \norm{\S_{DA}} > 2\sqrt{\norm{\S_A}}\,\eta(k,n,p,\gamma_1))\\
& \leq \sum_{D\subset \comp{A}, |D|\leq k} \Pr(\norm{\S_{DA}} > 2\sqrt{\norm{\S_A}}\,\eta(k,n,p,\gamma_1))\\
& \leq \sum_{D\subset \comp{A}, |D|\leq k} \Pr(\norm{\bfY} > 2\sqrt{n}\,\eta(k,n,p,\gamma_1) ) \\
& \leq \sum_{l=1}^k \binom{p-k}{l} \exp\left(-2\gamma_1 k\log\frac{\eexp p}{k}\right)\\
& \leq k \left(\eexp \frac{p}{k}\right)^{k(1-2\gamma_1)} \leq (\eexp p)^{1-\gamma_1/2},
\end{align*}
where the third inequality is due to the Davidson-Szarek bound. Combining the two bounds completes the proof.
 \end{proof}

\begin{lemma}
\label{lem:S-Asel-dev}
Let $\gamma_1\ge 3$. 
Suppose that $\frac{k}{n}\log\frac{\eexp p}{k}\leq c_0$ for a sufficiently small constant $c_0 > 0$.
Then with probability at least $1-12 (\eexp p)^{1-\gamma_1/2}$, $\bbB_k\neq \emptyset$ and
\begin{equation}
\norm{\S_{\Asel} - \S_A}\leq  C_0(\gamma_1) \sqrt{\lambda_1+1}\,\eta(k,n,p,\gamma_1),
	\label{eq:Ssel}
\end{equation}
where
$C_0(\gamma_1) = 14+4\sqrt{\gamma_1}$.
\end{lemma}
\begin{proof}
We focus on the event $\bbB_k \neq \emptyset$.
Define the sets
\begin{equation}
\label{eq:Asel-decomp}
G = A\cap \Asel, \quad M = A\cap \comp{\Asel},\quad  O = \comp{A}\cap \Asel \,,
\end{equation}
which correspond to the sets of correctly identified, missing, and overly identified features by the selected support set $\Asel$, respectively.
By the triangle inequality, we have
\begin{align}
	\label{eq:S-Asel-dev-decomp}
\norm{\S_{\Asel} - \S_{A}} \leq \norm{\S_{MM} - \bfI_{MM}} 
+ \norm{\S_{OO} - \bfI_{OO}} 
+ 2\norm{\S_{GM}}
+ 2\norm{\S_{GO}}.
\end{align}
We now provide high probability bounds for each term on the right-hand side.

For the first term, recall that $\Asel \in \bbB_k$ which is defined in \eqref{eq:supp-set}. Since $M\subset \comp{\Asel}$, we have
\begin{equation}
	\label{eq:S-Asel-dev-1}
\norm{\S_{MM} - \bfI_{MM}} \leq \eta(|M|,n,p,\gamma_1) \leq \eta(k,n,p,\gamma_1).
\end{equation}
For the second term, by similar calculation to that in the proof of \prettyref{lem:inclusion}, we have that when $\gamma_1\geq 3$, 
\begin{equation}
	\label{eq:S-Asel-dev-2}
\norm{\S_{OO} - \bfI_{OO}} \leq \eta(k,n,p,\gamma_1)
\end{equation}
with probability at least $1- 4(\eexp p)^{1-\gamma_1/2}$.
For the third term, we turn to the definition of $\bbB_k$ in \prettyref{eq:supp-set} again to obtain
\begin{align*}
\norm{\S_{GM}} \leq \norm{\S_{\Asel M}} 
\leq 2\sqrt{\norm{\S_{\Asel\Asel}}}\,\eta(k,n,p,\gamma_1).
\end{align*}
By \prettyref{lem:submat-opnorm-bd}, 
with probability at least $1-(\eexp p)^{1-\gamma_1/2}$,
\begin{align}
\norm{\S_{\Asel\Asel}},\norm{\S_{AA}} & \leq \max_{|B|=k}\norm{S_{BB}}
\leq (\lambda_1 + 1) 
\Big(1 + \sqrt{\frac{k}{n}} + \sqrt{\frac{\gamma_1 k}{n}\log\frac{\eexp p}{k}}\Big)^2 \nonumber \\
& \leq (1+(\sqrt{\gamma_1}+1)\sqrt{c_0})^2 (\lambda_1+1) \nonumber \\
& \leq (\tfrac{1}{2}\sqrt{\gamma_1}+\tfrac{3}{2})^2(\lambda_1+1), \label{eq:SAsel-norm}
\end{align}
where the second inequality holds when $\frac{k}{n}\log\frac{\eexp p}{k}\leq c_0$ and
the last inequality holds for sufficiently small constant $c_0$.
Moreover, the last two displays jointly imply
\begin{equation}
	\label{eq:S-Asel-dev-3}
\norm{\S_{GM}}\leq 
(\sqrt{\gamma_1}+3) 
\sqrt{\lambda_1+1}\,\eta(k,n,p,\gamma_1).
\end{equation}
For the fourth term, we obtain by similar arguments that
with probability at least $1-(ep)^{1-\gamma_1/2}$, 
\begin{equation}
	\label{eq:S-Asel-dev-4}
\norm{\S_{GO}}\leq \norm{\S_{A O}} \leq 
(\sqrt{\gamma_1}+3)
\sqrt{\lambda_1+1}\,\eta(k,n,p,\gamma_1).
\end{equation}

Note that $\Pr(\bbB_k = \emptyset)\leq \Pr(A\notin \bbB_k)$, and by \prettyref{lem:inclusion}, the latter is bounded above by $5(\eexp p)^{1-\gamma_1/2}$.
So the union bound implies that the intersection of the event $\{\bbB_k\neq \emptyset\}$ and the event that
\prettyref{eq:S-Asel-dev-1}--\prettyref{eq:S-Asel-dev-4} all hold has probability at least $1-12(\eexp p)^{1-\gamma_1/2}$.
On this event, we assemble 
\prettyref{eq:S-Asel-dev-decomp}--\prettyref{eq:S-Asel-dev-4} to obtain
\begin{align*}
\opnorm{\Ssel-\S_{A}} & \leq [2 + 4(\sqrt{\gamma_1}+3)\sqrt{\lambda_1+1}]\,\eta(k,n,p,\gamma_1)\\
& \leq [2 + 4(\sqrt{\gamma_1}+3)]\,\sqrt{\lambda_1+1}\,\eta(k,n,p,\gamma_1).
\end{align*}
This completes the proof.
 \end{proof}

\begin{lemma}
\label{lmm:wishart-bd}
Let $\Y\in \RR^{n\times k}$ and $\Z\in \RR^{n\times m}$ be two independent matrices with i.i.d.~$N(0,1)$ entries. Then there exists an absolute constant $C>0$, such that
\begin{align}
&	\Ex \Big\|\frac{1}{n}\Y'\Y - \I_k\Big\|^2 \leq  C \pth{\frac{k}{n} + \frac{k^2}{n^2}},\quad \mbox{and}
	\label{eq:YY}\\
&	\Ex \opnorm{\Y'\Z}^2 \leq  C (n(k+m) + km).
	\label{eq:YZ}
\end{align}
\end{lemma}
\begin{proof}
The inequality \prettyref{eq:YY} follows directly from integrating the high-probability upper bound in \cite[Proposition 4]{CMW12}.

	Let the SVD of $\Y$ be $\Y=\A\C\B'$, where $\A,\B$ are $n \times (n\wedge k)$ and $k \times (n\wedge k)$ uniformly Haar distributed, and $\C$ is an $(n\wedge k) \times (n\wedge k)$ diagonal matrix with $\|\C\| = \|\Y\|$. 
	Since $\A$ and $\Z$ are independent, $\A'\Z$ has the same law as a $(n\wedge k) \times m$ iid Gaussian matrix. Therefore
	\[
	\expect{\|\Y'\Z\|^2} \leq \expect{\|\Y\|^2} \expect{\|\A'\Z\|^2} \leq C_0 (n+k) (n\wedge k + m).
	\]
	for some absolute constant $C_0$, where the last inequality follows from the Davidson-Szarek theorem \cite[Theorem II.7]{Davidson01}. 	Exchanging the role of $\Y$ and $\Z$, we have $\expect{\|\Y'\Z\|^2} \leq C_0 (n+m) (n\wedge m + k)$. Consequently, 
	\[
	\expect{\|\Y'\Z\|^2} \leq C_0 ((n+k) (n\wedge k + m) \wedge (n+m) (n\wedge m + k)) \leq 2 C_0 (n(k+m) + km).
	\]
This completes the proof.
 \end{proof}

In the proofs, the following \emph{intermediate matrix}
\begin{equation}
\S_0 = \frac{1}{n} \V\D\U'\U\D\V'+\I_p
\label{eq:S0}
\end{equation}
plays a key role. 
In particular, the following results on $\S_0$ will be used repeatedly.
\begin{lemma}
\label{lem:S0}
Suppose $\lambda_1/\lambda_r\leq \tau$ for some constant $\tau\geq 1$.
If $\opnorm{\frac{1}{n}\U'\U - \I_r} < 1/\tau$, then
\begin{equation}
\label{eq:S0-eval}
\sigma_l(\S_0)>1,\quad \forall l\in [r], \quad \mbox{and}\quad \sigma_l(\S_0)=1, \quad  \forall l>r. 
\end{equation}
Moreover, $\V\V'$ is the projection matrix onto the rank $r$ principal subspace of $\S_0$.
\end{lemma}
\begin{proof}
It is straightforward to verify that
\begin{equation}
\S_0-\bfSigma = \V\D\pth{\frac{1}{n}\U'\U-\I_r}\D\V'. 	
\label{eq:S0SS}
\end{equation}
When $\opnorm{\frac{1}{n}\U'\U - \I_r} < 1$, 
$\opnorm{\S_0-\bfSigma} \leq \opnorm{\D}^2 \opnorm{\frac{1}{n}\U'\U - \I_r} < \lambda_1 / \tau \leq \lambda_r$.
So Weyl's inequality {\cite[Theorem 4.3.1]{HornJohnson}} leads to
\[
\sigma_r(\S_0) \geq \sigma_r(\bfSigma) - \opnorm{\S_0-\bfSigma} > \lambda_r+1-\lambda_r =1.
\]
Note that $\S_0$ always has $1$ as its eigenvalue with multiplicity at least $p-r$. We thus obtain \eqref{eq:S0-eval}.

When \eqref{eq:S0-eval} holds, \eqref{eq:S0} shows that the rank $r$ principal subspace of $\S_0$ is equal to that of $\frac{1}{n}\V\D\U'\U\D\V'$.
Therefore, the subspace is spanned by the column vectors of $\V$, and $\V\V'$ is the projection matrix onto it since $\V\in O(p,r)$. 
\end{proof}

%

To prove the lower bound for rank detection, we need the following lemma concerning the 
the $\chi^2$-divergences in covariance models. Recall that the $\chi^2$-divergence between two probability measures is defined as
\[
\chi^2(P \, || \, Q) \triangleq \int \pth{\fracd{P}{Q} - 1}^2 \diff Q.
\]
For a distribution $F$, we use $F^{\otimes n}$ to denote the product distribution of $n$ copies of $F$.
\begin{lemma}
Let $\nu$ be a probability distribution on the space of $p \times p$ symmetric 
random matrix $\M$ such that $\opnorm{\M} \leq 1$ almost surely. 
Consider the scale mixture distribution $\Expect[N(0,  \I_p + \M)^{\otimes n}] = \int N(0,  \I_p + \M)^{\otimes n} \nu(\diff \M)$. Then
\begin{align}
\chi^2(\Expect[N(0,  \I_p + \M)^{\otimes n}] \, || N(0,  \I_p)^{\otimes n})
= & ~ \expect{\det(\I_p - \M_1 \M_2)^{-\frac{n}{2}}}  -1 
\label{eq:chi2-cov}
\end{align}
where $\M_1$ and $\M_2$ are independently drawn from $\nu$. 

	\label{lmm:chi2.cov}
\end{lemma}
\begin{proof}
	Denote by $g_{i}$ the probability density function of $N\left(0,\bfSigma _{i}\right) $ for $i=0,1$ and $2$, respectively. Then it is straightforward to verify that
\begin{align}
\int \frac{g_{1}g_{2}}{g_{0}}
= & ~ \det(\bfSigma_0)^{1/2} \left[ \det \left( \bfSigma _{1}+\bfSigma _{2} - \bfSigma_1\bfSigma_{0}^{-1} \bfSigma_2\right)\right] ^{-\frac{1}{2}} \label{eq:g012}
\end{align}
if $ \bfSigma _{1}+\bfSigma _{2} \geq \bfSigma_1\bfSigma_{0}^{-1} \bfSigma_2$; otherwise, the integral on the left-hand side of \prettyref{eq:g012} is infinite.
Applying \prettyref{eq:g012} to $\bfSigma_0=\I_p$ and $\bfSigma_i=\I_p+\M_i$ and using Fubini's theorem yield \prettyref{eq:chi2-cov}.
\end{proof}

\subsection{Proofs of Main Results}
\label{sec:pf-main}


\subsubsection{Proofs of the Upper Bounds}
{\bf Proof of \prettyref{thm:rate-Sigma}}
Let $\gamma_1 \geq 10$ be a constant.
Denote by $E$ the event that \prettyref{eq:Ssel} holds, which, in particular, contains the event $\{\bbB_k \neq \emptyset\}$. By triangle inequality,
	\begin{align}
\expect{\norm{\S_{\Asel} - \bfSigma}^2 \ones_E}
\leq & ~ 2 \, \expect{\norm{\S_{\Asel} - \S_A}^2 \ones_E} + 2 \, \expect{\norm{\S_A - \bfSigma}^2 \ones_E} 	\nonumber \\
\lesssim & ~ 2 (\lambda_1 +1) \eta^2(k,n,p,\gamma_1)	 + 2 \, \expect{\norm{\S_A - \bfSigma}^2} \nonumber \\
\asymp & ~ (\lambda_1 + 1) \frac{k}{n} \log\frac{\eexp p}{k}	 + \expect{\norm{\S_A - \bfSigma}^2} , \label{eq:SAhatSS}
\end{align}	
where the last step holds because $\frac{k}{n}\log\frac{\eexp p}{k}\leq c_0$.
In view of the definition of $\eta$ in \prettyref{eq:diag-dev}, we have $\eta^2(k,n,p,\gamma_1) \asymp \frac{k}{n} \log\frac{\eexp p}{k}$.

To bound the second term in \prettyref{eq:SAhatSS}, define $\J_B$ as the diagonal matrix given by
\begin{equation}
(\J_B)_{ii} = \ind{i\in B}.	
	\label{eq:JB}
\end{equation}
Then, for $\S_0$ in \eqref{eq:S0},
	\begin{align}
	\S_A 
= & ~ \J_A (\S - \I) \J_A + \I	
= \J_A \S  \J_A - \I_{AA} + \I	\nonumber \\
= & ~ \S_0 + \frac{1}{n} \V \D \U' \Z \J_A  + \frac{1}{n} \J_A \Z' \U \D\V'  +  \J_A \pth{\frac{1}{n} \Z'\Z-\I}\J_A.
\label{eq:SA}
\end{align}
	Therefore
	\begin{align}
	\|\S_A - \bfSigma\|
\leq & ~ \|\S_0 - \bfSigma\| + \frac{2}{n} \norm{\V \D \U' \Z \J_A}  +  \frac{1}{n} \Opnorm{\J_A \pth{\frac{1}{n} \Z'\Z-\I}\J_A}
\label{eq:SASS}
\end{align}
In view of \prettyref{eq:S0SS} and \prettyref{lmm:wishart-bd}, we have
\begin{equation}
	\Ex \|\S_0 - \bfSigma\|^2 \leq \norm{\D}^4 \, \Ex \opnorm{\frac{1}{n}\U'\U-\I_r}^2 \lesssim \lambda_1^2 \pth{\frac{r}{n} + \frac{r^2}{n^2}} \asymp \frac{\lambda_1^2 r}{n},
	\label{eq:SASS-1}
\end{equation}
where the last step is due to $n \geq c_0 k \log \frac{\eexp p}{k} \geq c_0 k$ by assumption.
Similarly, 
\begin{equation}
	\Ex \Opnorm{ \J_A \pth{\frac{1}{n} \Z'\Z-\I}\J_A}^2 \lesssim \frac{k}{n},
	\label{eq:SASS-2}
\end{equation}
Again by \prettyref{eq:YZ} in \prettyref{lmm:wishart-bd},
\begin{equation}
	\frac{1}{n^2} \Ex \norm{\V \D \U' \Z \J_A}^2 \leq \frac{\lambda_1}{n^2} \Ex \norm{\U' \Z \J_A}^2 \lesssim \frac{\lambda_1 (n(k+r)+kr)}{n^2} \asymp \frac{\lambda_1 k}{n} .
	\label{eq:SASS-3}
\end{equation}
Assembling \prettyref{eq:SASS} -- \prettyref{eq:SASS-3}, we have
\begin{equation}
	\Ex \|\S_A - \bfSigma\|^2 \lesssim \frac{\lambda_1^2 r}{n} + \frac{(\lambda_1+1) k}{n} .
	\label{eq:SASS-4}
\end{equation}
	Combining \prettyref{eq:SAhatSS} and \prettyref{eq:SASS-4} yields
	\begin{equation}
	\expect{\norm{\S_{\Asel} - \bfSigma}^2 \ones_E} \lesssim (\lambda_1+1)  \frac{k}{n} \log\frac{\eexp p}{k} + \frac{\lambda_1^2 r}{n}.
	\label{eq:error-E}
\end{equation}

Next we control the estimation error conditioned on the complement of the event $E$. First note that $\S_{\Asel}-\bfSigma = \begin{bmatrix}
\S_{\Asel \Asel} & \bszero \\ \bszero & \bfI_{\comp{\Asel}\comp{\Asel}}
\end{bmatrix} - \bfSigma$. Then $\|\Ssel - \bfSigma\| \leq \|\S\| + \|\bfSigma\| + 1 \leq \norm{\bfSigma}(\norm{\W_p} + 1) + 1$, where $\W_p$ is equal to $\frac{1}{n}$ times a $p \times p$ Wishart matrix with $n$ degrees of freedom. Also, $\norm{\bfSigma-\I} = \lambda_1$. 
In view of \prettyref{eq:hSS}, we have $\norm{\hSS-\bfSigma} \leq (1+\lambda_1)(\norm{\W_p} + 2)$. Using Cauchy-Schwartz inequality, we have
	\begin{align*}
\expect{\norm{\hSS - \bfSigma}^2 \ones_{\comp{E}}}
\leq & ~ (1+\lambda_1)^2 \expect{(\norm{\W_p} + 2)^2 \ones_{\comp{E}}	}\\
\leq & ~ (1+\lambda_1)^2 \sqrt{\expect{(\norm{\W_p} + 2)^4}}  \sqrt{\prob{\ones_{\comp{E}}	}}.
\end{align*}
Note that $\|\W_p\| \eqdistr \frac{1}{n} \sigma_1^2(\Z)$. By  \cite[Theorem II.7]{Davidson01}, $\prob{\sigma_1(\Z) \geq 1 + \sqrt{\frac{p}{n}} + t} \leq \exp(-nt^2/2)$. Then $\expect{\norm{\W_p}^4} \lesssim \frac{1}{n^4} (1 + \frac{p^4}{n^4})$. By \prettyref{lem:S-Asel-dev}, we have $\prob{\comp{E}} \leq 12(\eexp p)^{1-\gamma_1/2}$. Therefore
\begin{equation}
\expect{\norm{\hSS - \bfSigma}^2 \ones_{\comp{E}}} \lesssim (1+\lambda_1)^2 \frac{1}{n^2} \pth{1 + \frac{p^2}{n^2} } p^{\frac{1}{2}-\frac{\gamma_1}{4}}
	\label{eq:error-Ec}
\end{equation}
Choose a fixed $\gamma_1 \geq 10$. Assembling \prettyref{eq:error-E} and \prettyref{eq:error-Ec}, we have
\begin{align*}
\expect{\norm{\hSS - \bfSigma}^2} 
= & ~ \expect{\norm{\S_{\Asel} - \bfSigma}^2 \ones_E} + \expect{\norm{\S_{\Asel} - \bfSigma}^2 \ones_{\comp{E}}} 	\\
\lesssim & ~ (1+\lambda_1)  \frac{k}{n} \log\frac{\eexp p}{k} + \frac{\lambda_1^2 r}{n} + (1+\lambda_1)^2 \frac{1}{n^2} 
\asymp 
(1+\lambda_1)  \frac{k}{n} \log\frac{\eexp p}{k} + \frac{\lambda_1^2 r}{n}.
\end{align*}	

\bigskip

\noindent{\bf Proof of \prettyref{thm:rate-rank}}
To prove the theorem,
it suffices to show that with the desired probability, 
\begin{equation}
\label{eq:rank-uppbd-equiv}
\sigma_r(\Ssel) > 1 + \gamma_2\,\sqrt{\opnorm{\Ssel}}\,\eta(k,n,p,\gamma_1) > \sigma_{r+1}(\Ssel).
\end{equation}

Recall \eqref{eq:S0} and \eqref{eq:S0SS}.
By \cite[Proposition 4]{CMW12}, with probability at least $1-2(\eexp p)^{-\gamma_1/2}$, $\opnorm{\frac{1}{n}\U'\U -\I_r} \leq \eta(k,n,p,\gamma_1)$.
Under \eqref{eq:thm-rank-assumpt}, $\sqrt{\frac{k}{n}\log\frac{\eexp p}{k}}\asymp \eta(k,n,p,\gamma_1)\leq 1/(2\tau)$ when $c_0$ is sufficiently small, 
and so
$\eta(k,n,p,\gamma_1)\leq 1/(2\tau)$. 
Thus,
\begin{equation}
	\label{eq:S0-dev}
\opnorm{\S_0 - \bfSigma} 
\leq \opnorm{\D}^2 \opnorm{\frac{1}{n}\U'\U -\I_r}
\leq \lambda_r/2.
\end{equation}
Therefore, \prettyref{lem:S0} leads to \eqref{eq:S0-eval}.
Moreover, Weyl's inequality leads to
\begin{align}
	\label{eq:sigma-r-S0}
\sigma_r(\S_0) \geq \sigma_r(\bfSigma) - |\sigma_r(\S_0) - \sigma_r(\bfSigma)| 
\geq \lambda_r + 1- \lambda_1\eta(k,n,p,\gamma_1)\geq \hf\lambda_r + 1 > 1.
\end{align}
Next, we consider $\S_A - \S_0$. By \prettyref{eq:SA}, we have
\[
	\|\S_A - \S_0\|
\leq \frac{2}{n} \norm{\V \D \U' \Z \J_A}  +  \frac{1}{n} \norm{\J_A (\Z'\Z-\I)\J_A}.
\]
By \cite[Proposition 4]{CMW12} and \eqref{eq:thm-rank-assumpt}, when $c_0$ is sufficiently small, with probability at least $1-(\eexp p + 1)(\eexp p)^{-\gamma_1/2}$,
\begin{align*}
\norm{\V \D \U' \Z \J_A} & \leq \opnorm{\D}\opnorm{\U'\Z\J_A} \\
& \leq \sqrt{\lambda_1}\, n\sqrt{1+\frac{8\gamma_1\log p}{3n}}
\Big(\sqrt{\frac{r}{n}}+\sqrt{\frac{k}{n}} + \sqrt{\gamma_1\log \frac{\eexp p}{n}}\Big)\\
& \leq n \sqrt{\lambda_1}\,  \eta(k,p,n,\gamma_1).
\end{align*}
Moreover, \cite[Proposition 3]{CMW12} implies that with probability at least 
$1-2 (\eexp p)^{-\gamma_1/2}$, 
\[
\frac{1}{n} \norm{\J_A (\Z'\Z-\I)\J_A} \leq \eta(k,n,p,\gamma_1).
\]
Assembling the last three displays, we obtain that with probability at least $1-(\eexp p+3) (\eexp p)^{-\gamma_1/2}$, 
\begin{equation}
	\label{eq:SA-S0-dev}
	\opnorm{\S_A - \S_0} \leq (2\sqrt{\lambda_1}+1)\,\eta(k,n,p,\gamma_1).
\end{equation}
Last but not least, \prettyref{lem:S-Asel-dev} implies that, 
under the condition of \prettyref{thm:rate-rank}, 
with probability at least $1-12(\eexp p)^{1-\gamma_1/2}$, \eqref{eq:Ssel} holds.
Together with \eqref{eq:SA-S0-dev}, it implies that
\begin{equation}
	\label{eq:SAsel-S0-dev}
\opnorm{\Ssel-\S_0}
\leq 
[C_0(\gamma_1) +3 ]
\sqrt{\lambda_1+1}\,\eta(k,n,p,\gamma_1),
\end{equation}
where $C_0(\gamma_1) = 14 + 4\sqrt{\gamma_1}$. 
By \eqref{eq:thm-rank-assumpt}, we could further upper bound the right-hand side by $\lambda_1/4 \vee 1/2$.
When $\lambda_1 > 1$, $\sqrt{\lambda_1+1}\leq \sqrt{2\lambda_1} < \sqrt{2}\lambda_1$, so for sufficiently small $c_0$, \eqref{eq:thm-rank-assumpt} implies that the right side of \eqref{eq:SAsel-S0-dev} is further bounded by $\lambda_1/4$.
When $\lambda_1\leq 1$, $\sqrt{\lambda_1+1}\leq \sqrt{2}$, and so the right side of \eqref{eq:SAsel-S0-dev} is further bounded by $1/2$ for sufficiently small $c_0$. 

Thus, the last display, together with \eqref{eq:S0-dev}, implies
\begin{align}
\label{eq:Ssel-SS-dev}
\opnorm{\Ssel - \bfSigma} 
& \leq \lambda_1 \eta(k,n,p,\gamma_1)
 + [C_0(\gamma_1) +3 ]\sqrt{\lambda_1+1}\,\eta(k,n,p,\gamma_1) \\
& < \frac{1}{2}(\lambda_1+1).\nonumber
\end{align}
Here, the last inequality comes from the above discussion, and the fact that $\eta(k,n,p,\gamma_1) < 1/4$ for small $c_0$.
The triangle inequality further leads to 
\begin{equation}
	\label{eq:Ssel-norm}
\frac{1}{2}(\lambda_1+1) <  
\opnorm{\bfSigma} - \opnorm{\Ssel - \bfSigma}\leq 
\opnorm{\Ssel} \leq \opnorm{\bfSigma} + \opnorm{\Ssel - \bfSigma} 
< \frac{3}{2}(\lambda_1+1).
\end{equation}

Set $\gamma_2 \geq 2[C_0(\gamma_1)+3] = 8\sqrt{\gamma_1}+34$.
Then \eqref{eq:S0-eval}, \eqref{eq:SAsel-S0-dev} and \eqref{eq:Ssel-norm} jointly imply that,
with probability at least $1-(13\eexp p+5)(\eexp p)^{-\gamma_1/2}$, 
the second inequality in \eqref{eq:rank-uppbd-equiv} holds.
Moreover, 
\eqref{eq:Ssel-SS-dev} and the triangle inequality implies that,
with the same probability, 
\begin{align*}
\sigma_r(\Ssel) & \geq \sigma_r(\bfSigma) - \opnorm{\Ssel-\bfSigma}\\
& \geq 1 + \lambda_r - \lambda_1 \eta(k,n,p,\gamma_1)
 - [C_0(\gamma_1) +3 ]\sqrt{\lambda_1+1}\,\eta(k,n,p,\gamma_1)\\
& \geq 1 + \lambda_r\Big[1-\tau \eta(k,n,p,\gamma_1) - \frac{\sqrt{\lambda_1+1}}{2\lambda_1}\gamma_2 \tau \eta(k,n,p,\gamma_1)\Big]\\
& \geq 1 + \frac{3}{2}\gamma_2 (\lambda_1+1)\eta(k,n,p,\gamma_1).
\end{align*}
Here, the last inequality holds when $\lambda_r \geq \beta\sqrt{\frac{k}{n}\log\frac{\eexp p}{k}}$ for a sufficiently large $\beta$ which depends only on $\gamma_1$, $\gamma_2$ and $\tau$.
In view of \eqref{eq:Ssel-norm}, the last display implies the first inequality of \eqref{eq:rank-uppbd-equiv}. 
This completes the proof of the upper bound.

\bigskip

\noindent{\bf Proof of \prettyref{thm:rate-V}}
Let $E$ be the event such that \prettyref{lem:inclusion}, \prettyref{lem:S-Asel-dev}, the upper bound in \prettyref{thm:rate-rank}, and \eqref{eq:sigma-r-S0} hold.
Then $\Pr(\comp{E}) \leq C(\eexp p)^{1-\gamma_1/2}$.

On the event $E$, 
$\wh\bfSigma = \Ssel$.
Moreover,
\prettyref{lem:S0} shows that
$\V\V'$ is the projection matrix onto the principal subspace of $\S_0$, and
\prettyref{thm:rate-rank} ensures $\wh\V\in \reals^{p\times r}$.
Thus, \prettyref{prop:SS-V} leads to
\[
\begin{aligned}
\norm{\hSS-\S_0}\ind{E} & \geq \frac{1}{2} (\sigma_r(\S_0)- \sigma_{r+1}(\S_0)) \norm{\hV\hV'-\V\V'}\ind{E}\\
& \geq \frac{\lambda_r}{4}\norm{\hV\hV'-\V\V'}\ind{E},
\end{aligned}
\]
and so
\begin{align*}
\Ex\norm{\hV\hV'-\V\V'}^2\ind{E} \leq
 \frac{16}{\lambda_r^2}\,\Ex\norm{\hSS-\S_0}^2\ind{E}.
\end{align*}
To further bound the right-hand side of the last display, we apply \eqref{eq:SA}, \eqref{eq:SASS-2}, \eqref{eq:SASS-3}, \prettyref{lem:inclusion} and \prettyref{lem:S-Asel-dev} to obtain
\[
\begin{aligned}
\Ex\norm{\hSS-\S_0}^2\ind{E} & \leq \Ex\norm{\Ssel-\S_0}^2\ind{E} \\
& \lesssim \Ex\norm{\Ssel-\S_A}^2\ind{E} + \Ex\norm{\S_A-\S_0}^2
\lesssim (\lambda_1+1)\frac{k}{n}\log\frac{\eexp p}{k}.
\end{aligned}
\]
Together with the second last display, this implies 
\begin{equation}
\label{eq:V-rate-E}
\Ex \norm{\hV\hV'-\V\V'}^2\ind{E} \leq C\tau^2
\frac{k(\lambda_1+1)}{n \lambda_1^2}\log\frac{\eexp p}{k}.
\end{equation}

Now consider the event $\comp{E}$. Note that $\opnorm{\hV\hV' - \V\V'}\leq 1$ always holds. 
Thus, 
\begin{equation}
\label{eq:V-rate-Ec}
\Ex \norm{\hV\hV'-\V\V'}^2\ind{\comp{E}} \leq \Pr(\comp{E}) \leq C(\eexp p)^{1-\gamma_1/2} \lesssim \frac{\lambda_1+1}{n\lambda_1^2},
\end{equation}
where the last inequality holds under condition \eqref{eq:n-p-lambda-assumpt} for all $\gamma_1 \geq (1+\frac{2}{M_0})M_1$.
Assembling \eqref{eq:V-rate-E} and \eqref{eq:V-rate-Ec}, we obtain the upper bounds.


\subsubsection{Proofs of the Lower Bounds}
\label{sec:pf-lower}

{\bf Proof of \prettyref{thm:lb-Sigma-V}}
$1^\circ$
The minimax lower bound for estimating $\sp(\V)$ follows straightforwardly from previous results on estimating the leading singular vector, \ie, the rank-one case (see, \eg, \cite{Birnbaum12,Vu12}). The desired lower bound \prettyref{eq:lb-V} can be found in \cite[Eq.~(58)]{CMW12} in the proof of \cite[Proof of Theorem 3]{CMW12}.

$2^\circ$
Next we establish the minimax lower bound for estimating the spiked covariance matrix $\bfSigma$ under the spectral norm, which is considerably more involved.
Let $\Theta_1 = \Theta_1(k,p,r,\lambda)$.
In view of the fact that $a \wedge (b+c) \leq a \wedge b + a \wedge c \leq 2 (a \wedge (b+c))$ for all $a,b,c \geq 0$, it is convenient to prove the following equivalent lower bound
		\begin{equation}
\inf_{\hSS} \sup_{\bfSigma \in\Theta} \Ex \|\hSS -\bfSigma \|^2 
\gtrsim
\lambda^2 \wedge \pth{ \frac{(\lambda+1) k}{n} \log \frac{e p}{k}} + \lambda^2 \pth{1  \wedge  \frac{r}{n}}.
	\label{eq:rate-Sigma-equiv}
\end{equation}
To this end, we show that the minimax risk is lower bounded by the two terms on the right-hand side of \prettyref{eq:rate-Sigma-equiv} separately. In fact, the first term is the minimax rate in the rank-one case and the second term is the rate of the oracle risk when the estimator knows the true support of $\V$.

$2.1^\circ$  Consider the following rank-one subset of the parameter space
	\[
\Theta'=\{\bfSigma = \I_p + \lambda \v \v': \v \in B_0(k) \cap \bbS^{p-1} \} \subset \Theta_1.
\]
Then $\sigma_1(\bfSigma)=1+\lambda$ and $\sigma_2(\bfSigma)=1$. For any estimator $\hSS$, denote by $\hv$ its leading singular vector. Applying \prettyref{prop:SS-V} yields
\[
	\norm{\hSS-\bfSigma} \geq  \frac{\lambda}{2} \norm{\hv\hv'-\v\v'}.
\]
Then
\begin{align}
\inf_{\hSS} \sup_{\bfSigma \in\Theta'} \Ex \|\hSS -\bfSigma \|^2 
\geq & ~ \frac{\lambda^2}{4} \inf_{\hv} \sup_{\v \in B_0(k) \cap \bbS^{p-1}}  \Ex \norm{\hv\hv'-\v\v'}^2 \nonumber \\
\geq & ~ \frac{\lambda^2}{8} \inf_{\hv} \sup_{\v \in B_0(k) \cap \bbS^{p-1}}  \Ex \fnorm{\hv\hv'-\v\v'}^2 \label{eq:bir1} \\
\gtrsim & ~ \lambda^2  \pth{1 \wedge  \frac{(\lambda+1) k}{n \lambda^2} \log \frac{e p}{k}}, \label{eq:bir2} 
\end{align}
where \prettyref{eq:bir1} follows from $\rank(\hv\hv'-\v\v') \leq 2$ and \prettyref{eq:bir2} follows from the minimax lower bound in \cite[Theorem 2.1]{Vu12} (see also 
\cite[Theorem 2]{Birnbaum12}) for estimating the leading singular vector.


$2.2^\circ$ To prove the lower bound $\lambda^2 \pth{1 \wedge\frac{r}{n}}$, 
consider the following (composite) hypotheses testing problem:
\begin{equation}
H_{0}:  \bfSigma = \begin{bmatrix}
\pth{1+\frac{\lambda}{2}} \I_r& \bszero \\ \bszero & \I_{p-r}
\end{bmatrix}, ~\mbox{vs.}~
H_{1}:  \bfSigma = \begin{bmatrix}
\pth{1+\frac{\lambda}{2}} (\I_r + \rho \v\v') & \bszero \\ \bszero & \I_{p-r}
\end{bmatrix}, \v \in \bbS^{r-1},
	\label{eq:oracle-test}
\end{equation}
where $\rho = \frac{b \lambda}{\lambda+2} (1 \wedge \sqrt{\frac{r}{n}})$ with a sufficiently small absolute constant $0 < b < 1/2$. Since $r \in [k]$ and $\rho (1+\frac{\lambda}{2}) \leq \lambda$, both the null and the alternative hypotheses belong to the parameter set $\Theta_1$ defined in \eqref{para.space1}. Following the Le Cam's two-point argument \cite[Section 2.3]{Tsybakov09}, next we show that the minimal sum of Type-I and Type-II error probabilities of testing \prettyref{eq:oracle-test} is non-vanishing. Since any pair of covariance matrices in $H_0$ and $H_1$ differ in operator norm by at least $\rho (1+\frac{\lambda}{2}) = \frac{b \lambda}{2} (1 \wedge \sqrt{\frac{r}{n}})$, we obtain a lower bound of rate $\lambda^2 (1 \wedge \frac{r}{n})$.

To this end, let $\X$ consist of $n$ iid rows drawn from $N(0,\bfSigma)$, where $\bfSigma$ is either from $H_0$ or $H_1$. 
Since under both the null and the alternative, the last $p-r$ columns of $\X$ are standard normal and independent of the first $r$ columns, we conclude that the first $r$ columns form a sufficient statistic. Therefore the minimal Type-I+II error probability testing \prettyref{eq:oracle-test}, denoted by $\epsilon_n$, is equal to that of the following testing problem of dimension $r$ and sample size $n$:
\begin{equation}
	H_{0}:  \bfSigma = \I_r, \quad \mbox{versus} \quad
H_{1}:  \bfSigma = \I_r + \rho \v\v', \v \in \bbS^{r-1}.
	\label{eq:oracle-test-r}
\end{equation}
Recall that the minimax risk is lower bounded by the Bayesian risk. 
For any random vector $\u$ taking values in $\bbS^{r-1}$, denote by the $\Expect [N(0, \I_r + \rho \u\u')^{\otimes n}]$ the mixture alternative distribution with a prior equal to the distribution of $\u$. 
Applying \cite[Theorem 2.2 (iii)]{Tsybakov09} we obtain the following lower bound in terms of the $\chi^2$-divergence from the mixture alternative to the null:
\begin{align}
\epsilon_n 
\geq & ~ 1 - \pth{\chi^2\pth{\Expect[N(0,\I_r+\rho \u\u')^{\otimes n}] \, || \, N(0,\I_r)^{\otimes n}}/2 }^{1/2}
	\label{eq:oracle-test-pe}
\end{align}
Consider the unit random vector $\u$ with iid coordinates taking values in $\frac{1}{\sqrt{r}}\{\pm 1\}$ uniformly. Since $\rho \leq 1$, applying \prettyref{lmm:chi2.cov} yields
\begin{align}
1+ \chi^2\pth{\Expect[N(0,\I_r+\rho \u\u')^{\otimes n}] \, || \, N(0,\I_r)^{\otimes n}}
= & ~ \expect{\det(\I_p - \rho^2 \Iprod{\u}{\tu}\u\tu')^{-\frac{n}{2}}} \nonumber  \\
= & ~ \expect{(\I_p - \rho^2 \Iprod{\u}{\tu}^2)^{-\frac{n}{2}}} \label{eq:ttt1}\\
\leq & ~ \expect{\exp\pth{n \rho^2 \Iprod{\u}{\tu}^2}},  \label{eq:ttt2} \\
= & ~ \expect{\exp\pth{n \rho^2 G_r^2 / r^2}},  \label{eq:ttt} 
\end{align}
where \prettyref{eq:ttt1} follows from the matrix determinant lemma, \prettyref{eq:ttt2} follows the fact that $\log(1-x) \geq -2x$ for $x \in [0,1/2]$ and $\rho \leq b < 1/2$, and \prettyref{eq:ttt} follows from the definition that $G_r$ denotes the symmetric random walk on $\integers$ at $r\Th$ step defined in \prettyref{eq:Gm}. Since $\frac{n \rho^2}{r^2} \leq \frac{b^2}{r}$, choosing $b^2 \leq \frac{1}{\rateconst}$ as a fixed constant and applying \prettyref{lmm:rm} with $p=k=r$ (the non-sparse case), we conclude that
\begin{equation}
\chi^2\pth{\Expect[N(0,\I_p+\lambda \u\u')^{\otimes n}] \, || \, N(0,\I_p)^{\otimes n}} \leq g(b^2) - 1,
	\label{eq:chi2-oracle}
\end{equation}
where $g$ is given by in \prettyref{lmm:rm} and satisfies $g>1$. 
Combining \prettyref{eq:oracle-test-pe} -- \prettyref{eq:chi2-oracle}, we obtain the following lower bound for estimating $\bfSigma$:
		\begin{equation}
\inf_{\hSS} \sup_{\bfSigma \in\Theta} \Ex \|\hSS -\bfSigma \|^2 
\gtrsim \pth{1+\frac{\lambda}{2}}^2 \rho^2 \asymp \lambda^2 \pth{1 \wedge\frac{r}{n}}.
	\label{eq:bir3}
\end{equation} 


As we mentioned in \prettyref{sec:lower}, the rank-detection lower bound in \prettyref{thm:lb-rank} is a direct consequence of \prettyref{prop:rank01} concerning testing rank-zero versus rank-one perturbation, which we prove below.

\subsection{Proof of \prettyref{prop:rank01}}
\begin{proof}
Analogous to \prettyref{eq:oracle-test-pe}, any random vector $\u$ taking vales in $\bbS^{p-1} \cap B_0(k)$ gives the Bayesian lower bound 
\begin{equation}
\calE_n(k,p,\lambda) \geq 1 - \pth{\chi^2(\Expect N(0,\I_p+\lambda \u\u')^{\otimes n} || N(0,\I_p)^{\otimes n})/2}.	
	\label{eq:lecam}
\end{equation}
Put 
\[
t \triangleq \frac{n \lambda^2}{k^2} = \frac{\beta_0}{k} \log \frac{\eexp p}{k}.
\]
Let the random sparse vector $\u$ be defined in \prettyref{eq:prior-wu}.
In view of \prettyref{lmm:chi2.cov} as well as the facts that $\rank(\lambda \v\v')=1$ and $\opnorm{\lambda \v\v'}=\lambda \leq 1$, we have
\begin{align*}
1+ \chi^2(\Expect[N(0,\I_p+\lambda \u\u')^{\otimes n}] \, & || \, N(0,\I_p)^{\otimes n})\\
= & ~ \expect{\exp\pth{n \lambda^2 \Iprod{\u\u'}{\tu\tu'}^2}} \\
= & ~ \expect{\exp\pth{n \lambda^2 \iprod{\frac{1}{\sqrt{k}} \J_I \w}{\frac{1}{\sqrt{k}} \J_{\tI} \tw}^2}}  \\
= & ~ \expect{\exp\pth{t (\w' \J_{I\cap \tI} \tw')^2}}  \\ 
= & ~ \expect{\exp\pth{t G_H^2}},  \label{eq:s4}
\end{align*}
where in the last step we have defined $H \triangleq |I \cap \tI|  \sim \text{Hypergeometric}(p,k,k)$ and $\{G_m\}$ is the symmetric random walk on $\integers$ defined in \prettyref{eq:Gm}. Now applying \prettyref{lmm:rm}, we conclude that
\[
\chi^2\pth{\Expect[N(0,\I_p+\lambda \u\u')^{\otimes n}] \, || \, N(0,\I_p)^{\otimes n}} \leq g(\beta_0) - 1
\]
where $g$ is given by in \prettyref{lmm:rm} satisfying $g(0+)=1$. In view of \prettyref{eq:lecam}, we conclude that
\[
\calE_n(k,p,\lambda)  \geq \max\{0,1 - \sqrt{(g(\beta_0)-1)/2}\} \triangleq w(\beta_0).
\]
Note that the function $w$ satisfies $w(0+)=1$.
 \end{proof}

\subsection{Proof of Propositions \ref{prop:rate-Omega} and \ref{prop:rate-Lambda}}
	\label{sec:pf-cor}

We give here a joint proof of Propositions \ref{prop:rate-Omega} and \ref{prop:rate-Lambda}.
\begin{proof}
$1^\circ$ (Upper bounds) By assumption, we have $\lambda \asymp 1$ and $\frac{k}{n}\log\frac{\eexp p}{k}\leq c_0$. Since $r \in [k]$, applying \prettyref{thm:rate-Sigma} yields
\begin{equation}
\sup_{\bfSigma \in\Theta_1(k, p, r, \lambda)} \Ex \|\hSS -\bfSigma \|^2 
\lesssim
\frac{k}{n} \log \frac{\eexp p}{k}.	
	\label{eq:eee}
\end{equation}
Note that
	\begin{align}
	\norm{\hSS^{-1}-\bfSigma^{-1}}
\leq & ~ \big\|\hSS^{-1}\big\| \big\|\bfSigma^{-1}\big\| \norm{\hSS-\bfSigma}	\leq \frac{\norm{\hSS-\bfSigma}}{1-\norm{\hSS-\bfSigma}} 	,
\label{eq:ooo}
\end{align}
where the last inequality follows from $\sigma_p(\bfSigma)=1$ and Weyl's inequality. By Chebyshev's inequality, $\prob{\|\hSS-\bfSigma\| \geq t} \leq \frac{1}{t^2} \Ex \|\hSS-\bfSigma\|^2$. Let $E=\{\|\hSS-\bfSigma\| \leq \frac{1}{2}\}$. Then 
\begin{equation}
\prob{\comp{E}} \lesssim \frac{k}{n} \log \frac{\eexp p}{k}.	
	\label{eq:ee}
\end{equation}
 Moreover, again by Weyl's inequality, we have $E \subset \{\sigma_p(\hSS) \geq \frac{1}{2}\}$ hence $\hOO\indc{E} = (\hSS)^{-1} \indc{E}$ in view of \prettyref{eq:hOO}.  On the other hand, by definition we always have $\opnorm{\hOO} \leq 2$. Therefore
\begin{align}
\Ex \norm{\hOO-\OO}^2 
= & ~ \Ex \norm{\hSS^{-1}-\bfSigma^{-1}}^2 \indc{E} + \Ex \norm{\hOO-\OO}^2 \indc{\comp{E}} \nonumber \\
\leq & ~ \Ex \norm{\hSS^{-1}-\bfSigma^{-1}}^2 \indc{E} + 3\prob{\comp{E}} \nonumber \\
\leq & ~ 4 \Ex \norm{\hSS -\bfSigma}^2 + 3\prob{\comp{E}} \label{eq:omega1}\\
\lesssim & ~\frac{k}{n} \log \frac{\eexp p}{k}. \label{eq:omega2}
\end{align}
where \prettyref{eq:omega1} follows from \prettyref{eq:ooo} and \prettyref{eq:omega2} is due to \prettyref{eq:eee} and \prettyref{eq:ee}.

Next we prove upper bound for estimating $\E$. Recall that $\E+\I_p$ and $\hE+\I_p$ give the diagonal matrices formed by the ordered singular values of $\bfSigma$ and $\hSS$, respectively.
Similar to the proof of \prettyref{prop:rate-Lambda}, Weyl's inequality	implies that $\|\hE -\E \|  = \|\hE + \I_p -(\E + \I_p) \| = \max_{i} |\sigma_i(\hSS) - \sigma_i(\bfSigma)| \leq \opnorm{\hSS-\bfSigma}$, where for any $i\in[p]$, $\sigma_i(\bfSigma)$ is the $i\Th$ largest eigenvalue of $\bfSigma$.
In view of \prettyref{thm:rate-Sigma}, we have $\Ex \|\hE -\E\|^2 \leq \frac{k}{n} \log \frac{\eexp p}{k}$.

$2^\circ$ (Lower bounds) The lower bound follows from the testing result in \prettyref{prop:rank01}. 
Consider the testing problem \eqref{eq:HT.rank01}, then
both the null ($\bfSigma=\I_p$) and alternatives ($\bfSigma=\I_p + \lambda \v\v'$) are contained in the parameter space $\Theta_1$. By \prettyref{prop:rank01}, they cannot be distinguished with probability $1$ if $\lambda^2 \asymp \frac{k}{n} \log \frac{\eexp p}{k}$. The spectra differs at least $|\sigma_1(\I_p) - \sigma_1(\I_p + \lambda \v\v')| = \lambda$. By the Woodbury identity, $(\I_p + \lambda \v\v')^{-1} = \I_p - \frac{\lambda}{\lambda+1} \v\v'$, hence $\opnorm{\I_p - (\I_p + \lambda \v\v')^{-1}} = \frac{\lambda}{1+\lambda} \asymp \lambda$. The lower bound $\frac{k}{n} \log \frac{\eexp p}{k}$ is now completed by way of the usual two-point argument.
 \end{proof}


\subsection{Proof of \prettyref{prop:SS-V}}
	\label{app:SS-V}
	
\begin{proof}
Recall that $\sigma_1(\bfSigma) \geq \ldots \geq \sigma_p(\bfSigma) \geq 0$ denote the ordered singular values of $\bfSigma$. If $\sigma_r(\hSS) \leq \frac{\sigma_r(\bfSigma)+ \sigma_{r+1}(\bfSigma)}{2}$, then by Weyl's inequality,
we have $\norm{\hSS-\bfSigma} \geq \frac{1}{2} (\sigma_r(\bfSigma)- \sigma_{r+1}(\bfSigma))$.
 If $\sigma_r(\hSS) \geq \frac{\sigma_r(\bfSigma)+ \sigma_{r+1}(\bfSigma)}{2}$, then by David--Kahan's sin-theta theorem \cite{Davis70} (see also \cite[Theorem 10]{CMW12} 
we have
\[
\norm{\hV\hV'-\V\V'} \leq \frac{\norm{\hSS-\bfSigma}}{|\sigma_r(\hSS) - \sigma_{r+1}(\bfSigma)|} \leq \frac{2 \norm{\hSS-\bfSigma}}{\sigma_r(\bfSigma) - \sigma_{r+1}(\bfSigma)},
\]
completing the proof of \prettyref{eq:SS-V} in view of the fact that $\norm{\hV\hV'-\V\V'} \leq 1$.
 \end{proof}

\subsection{Proof of \prettyref{lmm:rm}}
\label{sec:pf-lmmrm}

\begin{proof}
First of all, we can safely assume that $p \geq 5$, for otherwise the expectation on the right-hand side of \prettyref{eq:rm} is obviously upper bounded by an absolutely constant. In the sequel we shall assume that
\begin{equation}
0 < a \leq \frac{1}{36} \, .
	\label{eq:arange}
\end{equation}
We use normal approximation of the random walk $G_m$ for small $m$ and use truncation argument to deal with large $m$. To this end, we divide $[p]$, the whole range of $k$, into three regimes.

\emph{Case I: Large $k$.} Assume that $\frac{p}{2} \leq k \leq p$. Then $t \leq \frac{2a \log(2 \eexp)}{p}$. 
By the following non-asymptotic version of the Tusn\'ady's coupling lemma (see, for example, \cite[Lemma 4, p. 242]{BM89}), for each $m$, there exists $Z_m \sim N(0,m)$, such that
\begin{equation}
|G_m| \stackrel{\rm s.t.}{\leq} |Z_m| + 2.
	\label{eq:tusnady}
\end{equation}
Since $H\leq p$, in view of \prettyref{eq:arange}, we have
\begin{align}
\expect{\exp\pth{t G_H^2}}
\leq & ~ \expect{\exp\pth{t G_p^2}} \leq \expect{\exp\pth{t (|Z_p|+2)^2}}	\nonumber \\
\leq & ~ \exp(8t) \expect{\exp\pth{2 t Z_p^2}}	\label{eq:NP1}\\
= & ~  \frac{\exp(8t)}{\sqrt{1-4 pt}} \leq \frac{\exp(7a)}{\sqrt{1 - 8 \log (2 \eexp) a}}  
\label{eq:case1}
\end{align}

\emph{Case II: Small $k$.} Assume that $1 \leq k \leq \log \frac{\eexp p}{k}$, which, in particular, implies that $k \leq p^{\frac{1}{3}} \wedge \frac{p}{2}$ since $p \geq 5$. Using $G_H^2 \leq H^2 \leq k H$, we have
\begin{align}
\expect{\exp\pth{t G_H^2}} 
\leq & ~ \expect{\pth{\frac{\eexp p}{k}}^{a H}} \leq \pth{1 + \frac{k}{p-k} \pth{\pth{\frac{\eexp p}{k}}^{a} - 1 }}^k \label{eq:hbin}\\
\leq & ~ \exp\sth{\frac{2 \sqrt{k}}{\sqrt{p}} \pth{\pth{\frac{\eexp p}{k}}^{a}-1}} \label{eq:hbin2}\\
\leq & ~ \exp(8a) \,, \label{eq:case2}
\end{align}
where \prettyref{eq:hbin} follows from the stochastic dominance of hypergeometric distributions by binomial distributions (see, \eg, \cite[Theorem 1.1(d)]{KM10}) 
\begin{equation}
\Hyper(p,k,k) \stackrel{\rm s.t.}{\leq} \text{Bin}\pth{k,\frac{k}{p-k}}, \quad k\leq \frac{p}{2}
	\label{eq:hyperbin}
\end{equation}
and the moment generating function of binomial distributions, \prettyref{eq:hbin2} is due to $(1+x)^k \leq \exp(kx)$ for $x \geq 0$ and $\frac{k^2}{p} \leq \sqrt{\frac{k}{p}}$, \prettyref{eq:case2} is due to 
\[
\sup_{y \geq 1} \frac{(\eexp y)^a-1}{\sqrt{y}} = 2 \sqrt{\eexp} a ((1-2 a))^{\frac{1}{2 a}-1} \leq 4 a, \quad \forall \; a \leq \frac{1}{2},
\]
and

\emph{Case III: Moderate $k$.} Assume that $\log \frac{\eexp p}{k} \leq k \leq \frac{p}{2}$. Define
\begin{equation}
A \triangleq 
\frac{k}{\sqrt{a} \log\frac{\eexp p}{k}}	\wedge k
	\label{eq:A8t}
\end{equation}
and write
\begin{equation}
\expect{\exp\pth{t G_H^2}}	= \expect{\exp\pth{t G_H^2} \indc{H \leq A}} + \expect{\exp\pth{t G_H^2} \indc{H > A}}.
\label{eq:s5}
\end{equation}
Next we bound the two terms in \prettyref{eq:s5} separately: For the first term, we use normal approximation of $G_m$.  
By \prettyref{eq:tusnady} -- \prettyref{eq:NP1}, for each fixed $m \leq A$, we have
\begin{align}
\expect{\exp\pth{t G_m^2}}
\leq & ~  \frac{\exp(8t)}{\sqrt{1-4 m t}}	\leq \frac{\exp(8a)}{\sqrt{1 - 4 \sqrt{a}}}, \label{eq:s6}
\end{align}
where the last inequality follows from $k \geq \log \frac{\eexp p}{k}$ and \prettyref{eq:A8t}. 

To control the second term in \prettyref{eq:s5}, without loss of generality, we assume that $A \leq k$, \ie, $A = \frac{k}{\sqrt{a} \log\frac{\eexp p}{k}} \geq \frac{1}{\sqrt{a}}$.
we proceed similarly as in \prettyref{eq:hbin} -- \prettyref{eq:case2}:
\begin{align}
& ~\expect{\exp\pth{t G_H^2} \indc{H > A}} \nonumber \\
\leq & ~ \expect{\exp\pth{tk H} \indc{H > A}} \nonumber \\
\leq & ~ \sum_{m > A} \exp(t k m) \pth{\frac{k}{p-k}}^m \pth{\frac{p-2k}{p-k}}^{k-m} \binom{k}{m} \label{eq:d0}\\
\leq & ~ \sum_{m > A} \exp\pth{a m \log \frac{p}{k} - m \log \frac{p}{2k} +  m \log \pth{ \sqrt{a} \eexp \log\frac{\eexp p}{k}}  } \label{eq:d1} \\
\leq & ~ \sum_{m > A} \exp\sth{- \pth{1-a-\frac{1}{\eexp}} m \log 2 +  m \log(2 \sqrt{a} \eexp) + \frac{m}{\eexp}}   \label{eq:d3}\\
\leq & ~ 7 \exp\pth{-\frac{1}{\sqrt{a}}}, \label{eq:d4}
\end{align}
where \prettyref{eq:d0} follows from \prettyref{eq:hyperbin},
\prettyref{eq:d1} follows from that $\binom{k}{m} \leq \pth{\frac{\eexp k}{m}}^m$, $p \geq 2k$ and $m \leq A$, \prettyref{eq:d3} follows from that $\eexp \log x \leq x$ for all $x \geq 1$, and \prettyref{eq:d4} is due to \prettyref{eq:A8t} and our choice of $a$ in \prettyref{eq:arange}. Plugging \prettyref{eq:s6} and \prettyref{eq:d4} into \prettyref{eq:s5}, we obtain
\begin{equation}
	\expect{\exp\pth{t G_H^2}} \leq \frac{\exp(8a)}{\sqrt{1 - 4 \sqrt{a}}} + 7 \exp\pth{-\frac{1}{\sqrt{a}}}.
	\label{eq:case3}
\end{equation}
We complete the proof of \prettyref{eq:rm}, with 
$g(a)$ defined as the maxima of the right-hand sides of \prettyref{eq:case1}, \prettyref{eq:case2} and \prettyref{eq:case3}.
\end{proof}

%

\end{document}